\documentclass[11pt]{article} \usepackage[T1]{fontenc}
\usepackage{amsmath,amssymb,amsthm,amsfonts,stmaryrd,rotating,mathabx}
\usepackage[latin9]{inputenc} 
\usepackage[a4paper]{geometry}
\usepackage[matrix,arrow,curve]{xy}
\usepackage[pdftex]{hyperref}
\usepackage[draft]{fixme}
\usepackage{pslatex}

\newcommand{\Dha}{D(H_{\halpha},\tilde{\mu})}
\newcommand{\Dhb}{D(H_{\hbeta},\tilde{\mu})}
\newcommand{\Da}{D(H_{\alpha},\tilde{\mu})}
\newcommand{\Db}{D(H_{\beta},\tilde{\mu})}

\newcommand{\rotimes}{\rtensor{\beta}{\tilde \mu}{\alpha}}
\newcommand{\sotimes}{\rtensor{\halpha}{\tilde\mu}{\beta}}

\newcommand{\rwotimes}{\rtensor{\alpha}{\tilde \mu}{\hbeta}}
\newcommand{\swotimes}{\rtensor{\beta}{\tilde\mu}{\alpha}}

\newcommand{\E}{E_{\psi}^{\dag}}

\newcommand{\F}{E_{\phi}^{\dag}}
\newcommand{\kE}[1]{|\E\rangle_{#1}}

\newcommand{\kF}[1]{|\F\rangle_{#1}}
\newcommand{\bF}[1]{\langle\F|_{#1}}

\newcommand{\op}{\mathrm{op}}
\newcommand{\co}{\mathrm{co}}
\newcommand{\ev}{\mathrm{ev}}

\theoremstyle{definition} 

\swapnumbers
\newtheorem{remark}{Remark}[subsection]
\newtheorem{remarks}[remark]{Remarks}

\newtheorem{examples}[remark]{Examples} 

\theoremstyle{plain}
\newtheorem{definition}[remark]{Definition}
\newtheorem{theorem}[remark]{Theorem}
\newtheorem{proposition}[remark]{Proposition}
\newtheorem{corollary}[remark]{Corollary}
\newtheorem{lemma}[remark]{Lemma}

\newtheorem*{assumption*}{Assumption}
\newcommand{\mfrak}[1]{\mathfrak{#1}}
\newcommand{\frakN}{\mfrak{N}}

\newcommand{\frakA}{\mathfrak{A}}

\newcommand{\complex}{\mathbb{C}}
\newcommand{\reals}{\mathbb{R}}
\newcommand{\integers}{\mathbb{Z}}

\DeclareMathOperator{\Id}{id}

\DeclareMathOperator{\lspan}{span}

\newcommand{\smalldiagram}{}

\newcommand{\gtimes}{\stackrel{\Gamma}{\otimes}}
\newcommand{\todot}{\tilde\otimes}
\newcommand{\otimesB}{\underset{B}{\otimes}}

\newcommand{\hDelta}{\hat{\Delta}}

\newcommand{\frakb}{\mathfrak{b}}

\newcommand{\Lt}[1]{\triangleleft}
\newcommand{\Rt}[1]{\triangleright}

\newcommand{\halpha}{\widehat{\alpha}}
\newcommand{\hbeta}{\widehat{\beta}}

\newcommand{\hA}{\hat{A}} \newcommand{\ha}{\hat{a}}

\newcommand{\rtensor}[3]{ {_{#1}\! \underset{#2}{\otimes}\!
{}_{#3}}}

\newcommand{\tensor}[1]{\underset{#1}{\otimes}}

\newcommand{\btensor}{\underset{\mu}{\otimes}}

\newcommand{\otimesm}{\tensor{\tilde{\mu}}}

\newcommand{\fibre}[3]{ {_{#1}\! \underset{#2}{\ast}\!
{}_{#3}}}

\newcommand{\frange}{\rtensor{E_{\psi}^{\dag}}{\frakb}{E_{\phi}^{\dag}}}

\newcommand{\fwsource}{\rtensor{E_{\psi}^{\dag}}{\frakb}{E_{\phi}^{\dag}}}
\newcommand{\fwrange}{\rtensor{E_{\phi}^{\dag}}{\frakb}{E_{\phi}}}

\newcommand{\tl}{\ensuremath \olessthan}
\newcommand{\tr}{\ensuremath \ogreaterthan}

\newcommand{\Hmu}{K} 
\newcommand{\Hnu}{H}

\newcommand{\wHone}{\Hnu \swotimes \Hnu \swotimes \Hnu}
\newcommand{\wHtwo}{\Hnu \rwotimes \Hnu \swotimes \Hnu}
\newcommand{\wHthree}{\Hnu \rwotimes \Hnu \rwotimes \Hnu}

\newcommand{\SU}{\mathrm{SU}_{q}(2)}

\newcommand{\frakh}{\mathfrak{h}}

\makeatother

\newcommand{\rA}{{_{r}A}}
\newcommand{\sA}{{_{s}A}}
\newcommand{\Ar}{A_{r}}
\newcommand{\As}{A_{s}}

\title{Measured quantum groupoids \\
  associated to \\ proper dynamical quantum groups}

\author{Thomas Timmermann\footnote{{Supported by the SFB 878
    ``Groups, geometry and actions'' funded by the DFG}} \\
University of Muenster \\ Einsteinstr.\ 62, 48149
  Muenster, Germany \\ timmermt@uni-muenster.de
}

\begin{document}
\maketitle
\begin{abstract}
  Dynamical quantum groups were introduced by Etingof and
  Varchenko in connection with the dynamical quantum
  Yang-Baxter equation, and measured quantum groupoids
  were introduced by Enock, Lesieur and Vallin in their
  study of inclusions of type $\mathrm{II}_{1}$ factors.  In this
  article, we associate to suitable dynamical quantum
  groups, which are a purely algebraic objects, Hopf
  $C^{*}$-bimodules and measured quantum groupoids on the
  level of von Neumann algebras.  Assuming invariant
  integrals on the dynamical quantum group, we first construct a
  fundamental unitary which yields
  Hopf bimodules on the level of $C^{*}$-algebras and von Neumann
  algebras.   Next, we assume properness of the dynamical quantum
  group and lift the integrals to the operator
  algebras. In a subsequent article, this construction shall
  be applied to the dynamical $\SU$ studied by Koelink and
  Rosengren.


\end{abstract}

\tableofcontents

\section*{Introduction}

Dynamical quantum groups were introduced by Etingof and
Varchenko as an algebraic framework for the study of the
dynamical quantum Yang-Baxter equation
\cite{etingof:book,etingof:qdybe,etingof:exchange}, a
variant of the Yang-Baxter equation arising in statistical
mechanics.  Every (rigid) solution of this equation has a
naturally associated tensor category of representations
which turns out to be equivalent to the category of
representations of some dynamical quantum group.  In the
case of the basic rational or basic trigonometric solution,
this dynamical quantum group can be regarded as a
quantization of the function algebra on some
Poisson-Lie-groupoid.  In general, it
can be regarded as a quantum groupoid and fits into the theory
of Hopf algebroids developed by Böhm and others
\cite{bohm:algebroids}.


Measured quantum groupoids were introduced by Enock, Lesieur
and Vallin \cite{enock:action,lesieur} to capture
generalized Galois symmetries of certain inclusions of type
$\mathrm{II}_{1}$ factors
\cite{enock:inclusions1,enock-nest:inclusions,nikshych:inclusions}.
Apart from this fundamental example in von Neumann algebra
theory, which was also considered in the algebraic setting
\cite{kadison:inclusions,szlachanyi:inclusions}, and from the
finite case, only few measured quantum groupoids have been
constructed and investigated yet
\cite{lesieur,vallin:matched-groupoids}. 

Up to now, connections between algebraic and
operator-algebraic approaches to quantum groupoids have only
been explored in the finite case
\cite{nikshych:algversions,schauenburg:comparison,vallin:finite}
and in the form of a few examples and constructions that
exist on both levels. The situation is very different in the
area of
quantum groups, where Woronowicz's theory of compact quantum
groups \cite{woronowicz} and van Daele's theory of
multiplier Hopf algebras with integrals
\cite{kustermans:algebraic,daele} form a bridge between the
algebraic and operator-algebraic approaches, combining the
computational convenience of the former with the power and
richness of the latter.

In this article, we associate to suitable dynamical quantum
groups, which are purely algebraic objects, Hopf
$C^{*}$-bimodules and measured quantum groupoids on the
level of von Neumann algebras. The main example of a
dynamical group we have in mind for application  is the dynamical $\SU$
studied by Koelink and Rosengren \cite{koelink:su2}, and in
a subsequent article, we want to study the construction for
this example in detail.

On the dynamical quantum groups, we have to impose several
assumptions.

First, we need a left- and a right-invariant
integral, which correspond to fiber-wise integration on a
groupoid, and a weight on the basis that is suitably
quasi-invariant, such that the resulting total integrals are
faithful, positive, and coincide. In the case of the
dynamical $\SU$, the left- and right-invariant integrals can
be obtained from a Peter-Weyl decomposition due to Koelink
and Rosengren \cite{koelink:su2}, while the quasi-invariant
weight on the basis can be chosen quite freely.  

Second, we assume the dynamical quantum group to be proper,
which is the natural analogue of compactness and unitality
for quantum groupoids, and to possess a specific approximate
unit in the base algebra. The dynamical $\SU$ even is
compact and thus satisfies this condition.

Third, we assume that the quasi-invariant weight on the
basis admits a bounded GNS-con\-struction. Like the first
condition,  this one is very natural.  In the case of
the dynamical $\SU$, the base algebra is formed by all
meromorphic functions on the plane and does not admit any
non-trivial bounded representations. To apply our
construction, one therefore has to change the base and
check that the Peter-Weyl decomposition persists.

Given these assumptions, the measured quantum groupoid is
constructed as follows.

The algebraic GNS-construction, applied to the
total integral on the dynamical quantum group, yields a
Hilbert space of square-integrable functions on the
dynamical quantum group together with a natural
representation by densely defined multiplication operators. To obtain a
$C^{*}$-algebra or von Neumann algebra, one has to show that
these multiplication operators are bounded. To prove this
and to lift the comultiplication to the resulting
$C^{*}$-algebra and von Neumann algebra, we proceed as in
the case of quantum groups \cite{timmermann:buch} and
construct a fundamental unitary which is
pseudo-multiplicative on the level of von Neumann algebras
and  $C^{*}$-algebras in the sense of 
\cite{vallin:pmu}  and
\cite{timmermann:cpmu}, respectively. The general theory of these
unitaries then yields completions of the dynamical quantum
group in the form a Hopf $C^{*}$-bimodule and a Hopf
von-Neumann bimodule, and simultaneously a Pontrjagin dual
in the same form. Finally, we extend the invariant integrals to
the level of operator algebras, using properness of the
dynamical quantum group and standard von Neumann algebra
techniques.

This article is organized as follows.

Section \ref{section:algebra} provides the algebraic
basics on dynamical quantum groups and integration that are
needed for the construction in Section
\ref{section:reduced}.  We first generalize the definition
of a dynamical quantum group or $\mathfrak{h}$-Hopf
algebroid, allowing the base to be non-unital, then consider
left- and right-invariant integrals on the total algebra and
quasi-invariant weights on the basis, and finally construct
a $*$-algebra related to the Pontrjagin dual. The main
result of this section is the existence of a modular
automorphism for the total integral, which follows from a
strong invariance property similarly as in the setting of
multiplier Hopf algebras \cite{daele}.

Section \ref{section:reduced} presents the construction of
the measured quantum groupoid outlined above. It uses Connes
spatial theory, in particular the relative tensor product of
Hilbert modules, and the $C^{*}$-algebraic analogue of that
construction \cite{timmermann:fiber}, and introduces the
necessary concepts along the way when they are needed.

We use standard notation and adopt the following
conventions.
All algebras will be over the ground field $\complex$ and we
do not assume the existence of a unit element.  
Given a vector space $V$  with a subset
$X\subseteq V$, we denote by $\langle X
\rangle\subseteq V$ the linear span and, if $V$ is normed,  by $[X]\subseteq V$ the closed linear span of $X$. 
Inner products on Hilbert spaces will be linear in the
second and anti-linear in the first variable.

\section{Dynamical quantum groups with integrals on the algebraic
  level}

\label{section:algebra}

This section summarizes and develops the basics on dynamical quantum
groups and integration used in this article.  Before turning
to  details, let us outline the main concepts.

A dynamical quantum group is a special quantum groupoid and
as such consists of an algebra $B$ called the basis, an
algebra $A$, commuting inclusions $r,s\colon B^{(\op)} \to
A$, and  a comultiplication, antipode and counit which are
in some sense fibered over  $r$ and $s$. What makes it special is that
the basis $B$ is commutative, that $r(B)$ and $s(B)$ are
 central in $A$ up to a twist which is controlled by
an action of a group $\Gamma$ on $B$ and a bigrading of $A$
by $\Gamma$, and that the target of the comultiplication
is given by a nice monoidal product $A\todot A$.

Integration on a quantum groupoid involves several
ingredients. The analogue of the left- or right-invariance
property of Haar measures on groups, Haar systems on
groupoids, and Haar weights on quantum groups can be
formulated for maps $A\to B$ that are linear with respect to
$r(B)$ or $s(B)$, respectively. To obtain a total
integration $A\to \complex$, such a partial integral $A\to
B$ has to be composed with a suitable functional $B\to
\complex$ that has to be compatible with the action of
$\Gamma$.

Let us now turn to details. We proceed as follows.

From the beginning, we assume all our algebras to possess an
involution but not necessarily a unit. We first recall
terminology concerning non-unital algebras
(\S\ref{subsection:non-unital}), then describe the monoidal
product $A\todot A$ (\S\ref{subsection:monoidal}), and
define dynamical quantum groups or, more precisely,
multiplier $(B,\Gamma)$-Hopf $*$-algebroids
(\S\ref{subsection:axioms}). Afterwards, we introduce and
study integrals (\S\ref{subsection:bimeasured}--\S\ref{subsection:measured}) and prove the existence of a
modular automorphism that controls the deviation of the
total integral from being a. Using integration, we finally
construct the dual $*$-algebra of a multiplier
$(B,\Gamma)$-Hopf $*$-algebroid (\S\ref{subsection:dual}).

\subsection{Preliminaries on non-unital algebras} \label{subsection:non-unital} To handle
non-unital algebras, we  use extra non-degeneracy
assumptions and multiplier algebras \cite[appendix]{daele:0}
which are recalled below.

Let $R$ be an algebra, not necessarily unital. Given a left
$R$-module $M$, we say that $R$ \emph{has local units for
  $M$} if for each finite subset $F \subseteq M$, there
exists some $r \in R$ such that $rm=m$ for all $m\in F$
\cite{daele:tools}.  The corresponding notion for right
$R$-modules is defined similarly. We say that $R$ \emph{has
  local units} if it has local units for $R$, regarded as a
left and as a right $R$-module.

 Let $R$ and $S$ be algebras with local units, let $N$ be an
 $R$-$S$-bimodule and assume that $R$ and $S$ have local
 units for $N$.  A \emph{multiplier} of $N$ is a pair
 $T=(T_{R},T_{S})$, where $T_{R} \colon R\to N$ is a left
 $R$-module map and $T_{S} \colon S\to N$ a right $S$-module
 map satisfying $T_{R}(r)s=rT_{S}(s)$ for all $r\in R,s\in
 S$. Given such a multiplier, we write $rT:=T_{R}(r)$ and
 $Ts:=T_{S}(s)$ for all $r\in R$, $s\in S$. We denote the
 set of all multipliers of $N$ by $M(N)$.
Clearly, $N$ embeds into $M(N)$ and $M(N)$ carries a natural
structure of an $R$-$S$-bimodule that is compatible with this
embedding. 

Regarding $R$ as an $R$-$R$-bimodule, $M(R)$
becomes an algebra via $r(TT'):=(rT)T'$ and $(TT')r:=T(T'r)$,
and $R$ embeds into $M(R)$ as an essential ideal. If $R$ is
a $*$-algebra, then so is $M(R)$, where
$rT^{*}r'=(r'{}^{*}Tr^{*})^{*}$ for all $r,r'\in R,T\in
M(R)$. 

The bimodule $N$ becomes an
$M(R)$-$M(S)$-bimodule via $r'(rns)s':=(r'r)n(ss')$ for all
$r'\in M(R),r\in R,n\in N,s\in S,s'\in M(S)$, and $M(N)$ is
an $M(R)$-$M(S)$-bimodule via $r (r'Ts') := ((rr')T)s'$ and
$(r'Ts')s = r'(T(s's))$ for all $r\in R, r'\in M(R), T\in
M(N), s\in S, s'\in M(S)$.

A homomorphism $\pi \colon R \to M(S)$
is {\em non-degenerate} if $\langle \pi(R)S\rangle =
S=\langle S\pi(R)\rangle$; in that case, it extends uniquely
to a homomorphism $M(R) \to M(S)$ which is
again denoted by $\pi$ (see \cite{daele:0}).

\subsection{The category of $(B,\Gamma)^{\ev}$-algebras} \label{subsection:monoidal}

Let $B$ be a  commutative $*$-algebra with
local units, let $\Gamma$ be a group that acts on $B$ on
the left, and let $e \in \Gamma$ be the unit.

A {\em $(B,\Gamma)$-module} is a $\Gamma$-graded
$B$-bimodule $V = \bigoplus_{\gamma \in \Gamma} V_{\gamma}$
for which $B$ has local units, where each $V_{\gamma}$ is a
$B$-bimodule and $v b = \gamma(b)v$ for all $v \in
V_{\gamma}, b \in B,\gamma \in \Gamma$.  A {\em morphism} of
$(B,\Gamma)$-modules $V$ and $W$ is a morphism of
$\Gamma$-graded $B$-bimodules.

  A {\em $(B,\Gamma)$-algebra} is a $\Gamma$-graded
  $*$-algebra $A = \bigoplus_{\gamma \in \Gamma} A_{\gamma}$
  which has local units in $A_{e}$ and is equipped with a
  $*$-homomor\-phism $B \to M(A)$ that turns $A$ into a
  $(B,\Gamma)$-module.  Such a $(B,\Gamma)$-algebra is {\em
    proper} if $B$ maps into $A$.

  Given a $(B,\Gamma)$-algebra $A$ and $\gamma\in \Gamma$,
  we denote by $M(A)_{\gamma} \subseteq M(A)$ the space of
  all multipliers $T\in M(A)$ satisfying $TA_{\gamma'} \subseteq
  A_{\gamma\gamma'}$ and $A_{\gamma'}T\subseteq
  A_{\gamma'\gamma}$ for all $\gamma'\in \Gamma$.  

A {\em
    morphism} of $(B,\Gamma)$-algebras $A$ and $C$ is a
  non-degenerate, $B$-linear $*$-homomorphism $\pi \colon A
  \to M(C)$ satisfying $\pi(A_{\gamma}) \subseteq
  M(C)_{\gamma}$ for all $\gamma\in \Gamma$.  Such a
  morphism is {\em proper} if it maps $A$ into $C$.

Using the extension of non-degenerate homomorphisms to
multipliers, one defines the composition of morphisms and
checks that $(B,\Gamma)$-algebras form a category.  

The tensor product $B\otimes B$ is a 
$*$-algebra with local units and a natural action of
$\Gamma\times \Gamma$. Replacing $(B,\Gamma)$ by
$(B,\Gamma)^{\ev}:=(B\otimes B,\Gamma\times \Gamma)$ in the
definition above, we obtain the category of all {\em
  $(B,\Gamma)^{\ev}$-algebras}.  

Let $A$ be a $(B,\Gamma)^{\ev}$-algebra. We call an element
$x\in A$ \emph{homogeneous} and write
$\partial_{x}=\gamma$, $\bar\partial_{x}=\gamma'$ if $x \in A_{\gamma,\gamma'}$ for some
$\gamma,\gamma' \in \Gamma$.  Thus,
$\partial_{x}\partial_{y} = \partial_{xy}$, $\bar\partial_{x}\bar\partial_{y} = \bar\partial_{xy}$ and
$\partial_{x^{*}} = \partial_{x}^{-1}$,
$\bar \partial_{x^{*}} = \bar \partial_{x}^{-1}$ for all
homogeneous $x,y \in A$.  Define $r=r_{A}\colon B\to M(A)$
and $s=s_{A}\colon B\to M(A)$ by
$r(b)a = (b\otimes 1)a$ and $s(b)a=(1\otimes b)a$ for all
$a\in A$, $b\in B$.  We write $\rA,\Ar,\sA,\As$ if we
 consider $A$ as a  $B$-module via left or right
 multiplication via $r$ or $s$, respectively.

Clearly, $B$ is a $(B,\Gamma)$-algebra and $B\otimes B$ is a
$(B,\Gamma)^{\ev}$-algebra with respect to the trivial
gradings.  Every $(B,\Gamma)$-algebra $A$ can be regarded as
a $(B,\Gamma)^{\ev}$-algebra, where $A_{(\gamma,\gamma)} =
A_{\gamma}$ and $A_{(\gamma,\gamma')} = 0$ whenever
$\gamma\neq \gamma'$, and $(b \otimes b')a=bb'a$ for all
$b,b'\in B$, $a\in A$. Conversely, every
$(B,\Gamma)^{\ev}$-algebra $A$ can be considered as a
$(B,\Gamma)$-algebra via $r \colon B\to M(A)$ and the
grading given by $A_{\gamma}:=\bigoplus_{\gamma'}
A_{\gamma,\gamma'}$, or via $s \colon B\to M(A)$ and the
grading given by $A_{\gamma'}:=\bigoplus_{\gamma}
A_{\gamma,\gamma'}$. We write $(A,r)$ and $(A,s)$,
respectively, to denote the resulting $(B,\Gamma)$-algebras.

Denote by $B\rtimes \Gamma$ the crossed product for the
action of $\Gamma$ on $B$, that is,
the universal algebra containing $B$ and $\Gamma$ such that
$e=1_{B}$ and $b\gamma \cdot b'\gamma' = b\gamma(b')
\gamma\gamma'$ for all $b,b'\in B$, $\gamma,\gamma' \in
\Gamma$. This is a $(B,\Gamma)$-algebra with respect to the
natural inclusion $B \to B\rtimes \Gamma$ and the involution
and grading given by $(b\gamma)^{*} = \gamma^{-1}b^{*}$ and
$(B\rtimes \Gamma)_{\gamma} = B\gamma$ for all $b\in B$,
$\gamma\in \Gamma$.


The \emph{fiber product} of $(B,\Gamma)^{\ev}$-algebras $A$
and $C$ is defined as follows. The
subalgebra
\begin{align*}
  A \stackrel{\Gamma}{\otimes} C &:=
  \sum_{\gamma,\gamma',\gamma'' \in \Gamma}
  A_{\gamma,\gamma'} \otimes C_{\gamma',\gamma''} \subset A
  \otimes C
\end{align*}
is a $(B,\Gamma)^{\ev}$-algebra, where $\partial_{a\otimes
  c} = \partial_{a}$, $\bar \partial_{a\otimes c}=\bar\partial_{c}$ for all $a\in
A$, $c\in C$ and $(r\times s)(b\otimes b') = r_{A}(b)
\otimes s_{C}(b')$ for all $b,b'\in B$.  Let $I\subseteq
M(A\gtimes C)$ be the ideal generated by $\{s_{A}(b)\otimes
1 - 1\otimes r_{C}(b) : b\in B\}$. Then the quotient
\begin{align*}
  A
\todot C:=A\gtimes C/(I (A\gtimes C))
\end{align*}
is a $(B,\Gamma)^{\ev}$-algebra again, called the fiber
product of $A$ and $C$.  Write $a\todot c$ for the image of
an element $a\otimes c$ in $A\todot C$.

The assignment $(A,C) \mapsto A\todot C$ is functorial,
associative and unital. Indeed, for all morphisms of $(B,\Gamma)^{\ev}$-algebras
  $\pi^{1} \colon A^{1}\to C^{1}$, $\pi^{2}\colon A^{2}\to
  C^{2}$, there exists a morphism
  \begin{align} \label{eq:bgf-functorial} \pi^{1}
    \tilde\otimes \pi^{2} \colon A^{1} \tilde\otimes A^{2}
    \to C^{1} \tilde\otimes C^{2}, \quad a_{1} \todot a_{2}
    \mapsto \pi^{1}(a_{1}) \todot \pi^{2}(a_{2});
  \end{align}
  for all $(B,\Gamma)^{\ev}$-alge\-bras $A,C,D$, there
  exists an isomorphism
  \begin{align} \label{eq:bg-associative} (A \tilde\otimes
    C) \tilde\otimes D \to A \tilde\otimes (C \tilde\otimes
    D), \quad (a \todot c) \todot d \mapsto a \todot (c
    \todot d),
  \end{align}
  and for each $(B,\Gamma)^{\ev}$-algebra $A$, there exist
  isomorphisms
  \begin{align} \label{eq:bg-unital} (B\rtimes \Gamma)
    \tilde\otimes A \to A, \ b\gamma \todot a \mapsto r(b)a,
    \qquad A \tilde\otimes (B\rtimes \Gamma) \to A, \ a
    \todot b\gamma \mapsto s(b)a.
  \end{align}
  These isomorphisms are compatible in a natural sense and
  endow the category of  $(B,\Gamma)^{\ev}$-algebras with
  a monoidal structure. From now on, we shall use them
  without further notice.

  The category of $(B,\Gamma)^{\ev}$-algebras carries
  automorphisms $(-)^{\op}$ and $(-)^{\co}$ such that for
  each $(B,\Gamma)$-algebra $A$ and each morphism $\phi
  \colon A\to C$, we have $A^{\co}=A$ as an algebra,
  $A^{\op}$ is the opposite $*$-algebra of $A$, that is, the
  same vector space with the same involution and reversed
  multiplication, and
\begin{align}
  (A^{\op})_{\gamma,\gamma'} &=
  A_{\gamma^{-1},\gamma'{}^{-1}} \text{ for all }
  \gamma,\gamma'\in \Gamma, & r_{A^{\op}} &= r_{A}, &
  s_{A^{\op}} &= s_{A}, & \phi^{\op} &= \phi, \label{eq:bg-op}
  \\
  (A^{\co})_{\gamma,\gamma'} &= A_{\gamma',\gamma}
  \text{ for all } \gamma,\gamma'\in \Gamma, &
  r_{A^{\co}} &= s_{A}, & s_{A^{\co}} &= r_{A}, &
  \phi^{\co} &= \phi. \label{eq:bg-co}
\end{align}
These automorphisms are involutive and commute, that is,
\begin{align*}
  (-)^{\op} \circ (-)^{\op} &= \Id, &
  (-)^{\co} \circ (-)^{\co} &= \Id, & (-)^{\op} \circ
(-)^{\co} = (-)^{\co} \circ (-)^{\op}.
\end{align*}
Furthermore, they are compatible with the monoidal structure as
follows. Given $(B,\Gamma)$-algebras $A,C$, there exist
isomorphisms $(A \todot C)^{\op} \to A^{\op} \todot
C^{\op}$ and $ (A \todot C)^{\co} \to C^{\co} \todot
A^{\co}$ given by $a \todot c \mapsto a \todot c$ and $a
\todot c \mapsto c \todot a$, respectively.  Moreover,
$(B\rtimes \Gamma)^{\co} = B\rtimes \Gamma$, there exists
an isomorphism $S^{B\rtimes \Gamma} \colon B\rtimes \Gamma
\to (B\rtimes \Gamma)^{\op}$, $b\gamma \mapsto \gamma^{-1}
b$, and all of these isomorphisms and the isomorphisms 
in \eqref{eq:bg-associative} and \eqref{eq:bg-unital} are compatible in a natural
sense.

\subsection{Multiplier $(B,\Gamma)$-Hopf $*$-algebroids}
\label{subsection:axioms}

We shall work with variants of the $\frakh$-Hopf algebroids
and $(B,\Gamma)$-Hopf $*$-algebroids considered in
\cite{etingof:qdybe,koelink:su2} and \cite{timmermann:free},
respectively, where the basis need no longer be
unital. These variants consist of a
$(B,\Gamma)^{\ev}$-algebra and a comultiplication, counit
and antipode, which will be introduced one after the
other. To quickly proceed to the main part of this article,
we postulate all the usual properties of these maps as
axioms and leave a study of the axiomatics for later.

Given a $(B,\Gamma)^{\ev}$-algebra $A$, we denote by
$\tilde M(A \todot A) \subseteq M(A\todot A)$ the set of all
 $T\in M(A\todot A)$ for which all products of the form
 \begin{align*}
  T(x \todot
1_{M(A)}), && (x \todot 1_{M(A)})T, & &
T (1_{M(A)} \todot y), && (1_{M(A)} \todot y)T
 \end{align*}
where $x\in A_{\gamma,e},y\in
A_{e,\gamma},\gamma\in \Gamma$, lie in $A  \todot A$. Evidently, $\tilde M(A
\todot A)$ is a $*$-subalgebra of $M(A\todot A)$.
\begin{definition}
  A \emph{comultiplication} on a $(B,\Gamma)^{\ev}$-algebra
  $A$ is a morphism $\Delta$ from $A$ to $A \todot A$
  satisfying $\Delta(A) \subseteq \tilde M(A \todot A)$ and
  $(\Delta \tilde\otimes \Id) \circ \Delta = (\Id
  \tilde\otimes \Delta) \circ \Delta$.  A {\em (proper)
    multiplier $(B,\Gamma)$-$*$-bial\-gebroid} is a (proper)
  $(B,\Gamma)^{\ev}$-algebra with a comultiplication.  A
  {\em morphism} of multiplier
  $(B,\Gamma)$-$*$-bial\-gebroids $(A,\Delta_{A})$,
  $(B,\Delta_{B})$ is a morphism $\phi$ from $A$ to $B$
  satisfying $\Delta_{B} \circ \phi = (\phi \todot \phi)
  \circ \Delta_{A}$.
\end{definition}

Let $(A,\Delta)$ be a multiplier $(B,\Gamma)$-$*$-bialgebroid.
We adopt the Sweedler notation and write
\begin{align*}
  \Delta(a) &=\sum a_{(1)} \todot
a_{(2)}, & (\Delta \todot \Id)(\Delta(a)) &=\sum a_{(1)} \todot a_{(2)} \todot
a_{(3)} = (\Id \todot \Delta)(\Delta(a))
\end{align*}
and so on for each $a\in A$. In general, $a_{(1)}$ and
$a_{(2)}$ do not stand for elements of $A$ because
$\Delta(a)$ need not lie in $A \todot A$, but only in
$\tilde M(A \todot A)$.  Therefore, this notation requires
extra care; see \cite{daele:0,daele:tools} for a detailed
explanation in the context of multiplier Hopf algebras. 

We shall need to form products of the form $\Delta(x)(1 \otimes y)$ or
$(y \otimes 1)\Delta(x)$ when $\partial_{y} \neq e$ or
$\bar \partial_{y} \neq e$, respectively,  which are defined as
follows.  The multiplication on $A \otimes A$ induces a
canonical $A \todot A$-$A \otimes A$-bimodule structure on
$\sA \otimesB \rA$, and a canonical $A \otimes A$-$A \todot
A$-bimodule structure on $\As \otimesB \Ar$.  We thus obtain
natural maps $_{s}M(A) \otimesB {_{r}M(A)}\to M(\sA \otimesB
\rA)$ and $M(A)_{s} \otimesB M(A)_{r} \to M(\As \otimesB
\Ar)$ and  define
\begin{align*}
  T_{1} &\colon \As \otimesB \sA \to \sA \otimesB \rA, \quad x
  \otimesB y \mapsto \Delta(x)(1 \otimesB y) = \sum x_{(1)}
  \otimesB x_{(2)}y, \\ T_{2} &\colon
  {A_{r} \otimesB {_{r}A}} \to \As \otimesB \Ar, \quad x \otimesB y
  \mapsto (x\otimesB 1)\Delta(y) =\sum xy_{(1)} \otimesB y_{(2)}.
\end{align*}
Similarly, one can define  the maps
\begin{align*}
 T_{3} &\colon \sA \otimesB \As \to \As \otimesB \Ar, \ x
  \otimesB y \mapsto (1 \otimesB y)\Delta(x), & T_{4} &\colon
  \rA \otimesB \Ar \to \sA \otimesB \rA, \ x \otimesB y \mapsto
  \Delta(y)(x \otimesB 1).  
\end{align*}
\begin{definition}
    A \emph{counit} for a multiplier
  $(B,\Gamma)$-$*$-bialgebroid $(A,\Delta)$ is a proper
  morphism of $(B,\Gamma)^{\ev}$-algebras $\epsilon \colon A
  \to B\rtimes \Gamma$ satisfying $(\epsilon \tilde\otimes
  \Id) \circ \Delta = \Id_{A} = (\Id \tilde\otimes \epsilon)
  \circ \Delta$.
\end{definition}
Let $(A,\Delta)$ be a multiplier
$(B,\Gamma)$-$*$-bialgebroid with counit
$\epsilon$.  Using the linear maps
\begin{align*}
  \sharp &\colon B\rtimes \Gamma \to B, \
  \sum_{\gamma}b_{\gamma} \gamma \mapsto \sum_{\gamma}
  b_{\gamma}, & \flat &\colon B\rtimes \Gamma \to B, \
  \sum_{\gamma} \gamma b_{\gamma} \mapsto
  \sum_{\gamma}b_{\gamma},
\end{align*}
we define $\epsilon^{\sharp},\epsilon^{\flat}\colon A\to B$
by $\epsilon^{\sharp} :=\sharp \circ \epsilon$ and
$\epsilon^{\flat} := \flat \circ \epsilon$.  Define 
$m_{r}\colon \Ar \otimesB \rA \to A$ and $m_{s} \colon
\As\otimesB \sA \to A$ by $\sum_{i} x_{i}\otimesB y_{i}
\mapsto \sum_{i} x_{i}y_{i}$.
\begin{remarks}
  \begin{enumerate}
  \item Clearly, $\epsilon(A_{\gamma,\gamma'}) \subseteq (B\rtimes
    \Gamma)_{\gamma,\gamma'} = 0$ whenever $\gamma,\gamma'
    \in \Gamma$ and $\gamma\neq \gamma'$.
  \item If $\epsilon'$ is a counit as well, then $\epsilon =
    \epsilon \circ (\Id \todot \epsilon') \circ \Delta=
    \epsilon' \circ (\epsilon \todot \Id) \circ \Delta =
    \epsilon'$.
  \item The condition $(\epsilon \tilde\otimes \Id) \circ
    \Delta = \Id_{A} = (\Id \tilde\otimes \epsilon) \circ
    \Delta$ is equivalent to the relations
    \begin{align*}
      \sum r( \epsilon^{\sharp}(x_{(1)}))x_{(2)}y&=xy  = \sum xy_{(1)}
      s( \epsilon^{\flat}(y_{(2)})) \quad \text{for all
      }x,y\in A,
    \end{align*}
    and hence to commutativity of the diagrams
    \begin{align*}
      \xymatrix@!=2pt{ && \sA \otimesB \rA
        \ar[rrd]^(0.7){\epsilon^{\sharp} \otimesB \Id}
        & & \\
        \As \otimesB \sA \ar[rru]^(0.4){T_{1}}
        \ar[rrrr]_{m_{s}} && &&
        A, } &&
      \xymatrix@!=2pt{ && \As \otimesB \Ar
        \ar[rrd]^(0.7){\Id \otimesB \epsilon^{\flat}}
        & & \\
        \Ar \otimesB \rA \ar[rru]^(0.4){T_{2}}
        \ar[rrrr]_{m_{r}} && &&
        A.}
    \end{align*}
    Furthermore, this condition is equivalent to the relations
    \begin{align*}
      \sum xy_{(2)}r(\epsilon^{\flat}(y_{(1)})) = xy = \sum
      s(\epsilon^{\sharp}(x_{(2)}))x_{(1)}y \quad \text{for
        all } x,y\in A.
    \end{align*}
  \end{enumerate}
\end{remarks}
 
The definition of the antipode involves the isomorphism
\begin{align*}
  \sigma_{A,A}
\colon (A \todot A)^{\co,\op} \to A^{\co,\op} \todot
A^{\co,\op}, \quad x\todot y \mapsto y\todot x.
\end{align*}
\begin{definition} \label{definition:antipode}
  An \emph{antipode} for a multiplier
  $(B,\Gamma)$-$*$-bialgebroid $(A,\Delta)$ with counit
  $\epsilon$ is an isomorphism $S\colon A \to A^{\co,\op}$
  of $(B,\Gamma)^{\ev}$-algebras that makes the following
  diagrams commute:
  \begin{gather*}
    \xymatrix@R=15pt@C=25pt{ \As \otimesB \sA \ar[r]^{T_{1}}
      \ar[d]_(0.55){\epsilon^{\flat} \otimesB \Id}& \sA
      \otimesB \rA
      \ar[d]^(0.55){S\otimesB \Id} \\
      A & \ar[l]^{m_{r}} \Ar\otimesB \rA }, \qquad
    \xymatrix@R=15pt@C=25pt{ \Ar \otimesB \rA \ar[r]^{T_{2}}
      \ar[d]_(0.55){\Id \otimesB \epsilon^{\sharp}}& \As
      \otimesB \Ar
      \ar[d]^(0.55){\Id\otimesB S} \\
      A & \ar[l]^{m_{s}} \As\otimesB \sA }, \\
    \xymatrix@R=15pt@C=25pt{ A \ar[rr]^{S} \ar[d]_{\Delta}&&
      A^{\co,\op}
      \ar[d]^{\Delta^{\co,\op}} \\
      A \todot A \ar[r]^(0.35){S\todot S} & A^{\co,\op}
      \todot A^{\co,\op} & \ar[l]_(0.45){\sigma_{A,A}} (A
      \todot A)^{\co,\op}.  }
  \end{gather*}
  A \emph{multiplier $(B,\Gamma)$-Hopf $*$-algebroid} is a
  multiplier $(B,\Gamma)$-$*$-bialgebroid with counit and
  antipode.
\end{definition}
\begin{examples}
  \begin{enumerate}
  \item The tensor product $B \otimes B$ is a multiplier $(B,\Gamma)$-Hopf
    $*$-algebroid, where $\Delta(b \otimes b') = (b \otimes
    1) \todot (1 \otimes b')$, $ \epsilon(b\otimes b') =
    bb'$, $S(b \otimes b') = b' \otimes b$ for
    all $b,b'\in B$.  
  \item The crossed product $B \rtimes \Gamma$ is a
    multiplier $(B,\Gamma)$-Hopf $*$-algebroid, where
    $\Delta(b \gamma) = b\gamma \todot \gamma = \gamma
    \todot b\gamma$, $\epsilon=\Id$ and
    $S(\gamma b) = b\gamma^{-1}$ for all $b\in B,\gamma\in
    \Gamma$.
\end{enumerate}
\end{examples}
Given an antipode $S$ on a multiplier $(B,\Gamma)$-$*$-bialgebroid
$(A,\Delta)$ and an element $a\in A$, we shall henceforth
always regard $S(a)$ as an element of $A$ and not of
$A^{\co,\op}$.
\begin{remarks} \label{remarks:antipode} Let
  $(A,\Delta,\epsilon,S)$ be a multiplier $(B,\Gamma)$-Hopf
  $*$-algebroid.
  \begin{enumerate}
  \item In Sweedler notation, commutativity of the diagrams
    in Definition \ref{definition:antipode}  amounts to
    \begin{gather}
      \begin{aligned}
        \sum S(x_{(1)})x_{(2)}y &= s(\epsilon^{\flat}(x))y,
        & \sum xy_{(1)}S(y_{(2)}) &=
        xr(\epsilon^{\sharp}(y)) & \text{for all } x,y\in A,
      \end{aligned} \\ \label{eq:antipode-sigma}
      \sum S(x_{(1)}) \todot S(x_{(2)}) = \sum S(x)_{(2)}
      \todot S(x)_{(1)} \quad \text{for all } x\in A.
    \end{gather}
  \item If $S'$ is an antipode as well, then $S'=S$ because
    for all $x,y,z\in A$,
    \begin{align*}
      xS(y)z = S(yS^{-1}(x))z &= \sum S(s(
      \epsilon^{\sharp}(y_{(2)}))y_{(1)}S^{-1}(x))z  \\
      &= \sum S(y_{(2)}S^{-1}(x))
      r(\epsilon^{\sharp}(y_{(2)})) z \\ &= \sum
      S(y_{(1)}S^{-1}(x)) y_{(2)} S'(S'{}^{-1}(z)y_{(3)}) 
      =xS'(y)z.
    \end{align*}
  \end{enumerate}
\end{remarks}
For every multiplier $(B,\Gamma)$-Hopf $*$-algebroid, the
maps $T_{1}$ and $T_{2}$ defined above are bijections.
\begin{proposition} \label{proposition:galois} Let
  $(A,\Delta)$ be a multiplier $(B,\Gamma)$-$*$-bialgebroid.
  If $(A,\Delta)$ has a counit $\epsilon$ and an antipode
  $S$, then the maps $T_{1},T_{2},T_{3},T_{4}$ are bijective
  and for all $x,y\in A$,
  \begin{align*}
    T_{1}^{-1}(x\otimesB
  y)  &= \sum x_{(1)}\otimesB S(S^{-1}(y)x_{(2)}), &
  T_{2}^{-1}(x\otimesB y) &= \sum S(y_{(1)}S^{-1}(x))
  \otimesB y_{(2)}, \\
  T_{3}^{-1}(x \otimesB y) &= \sum x_{(1)} \otimesB S^{-1}(x_{(2)}S(y)), &
  T_{4}^{-1}(x \otimesB y) &= \sum S^{-1}(S(x)y_{(1)})
  \otimesB y_{(2)}.
  \end{align*}
\end{proposition}
\begin{proof}
  We only prove the assertion concerning $T_{1}$.  One first checks
  that the formula given for $T_{1}^{-1}$ yields a
  well-defined map $T_{1}' \colon \sA \otimesB \rA \to
  \As\otimesB \sA$, and then that for all $x,y\in A$ and
  $u,v \in A_{e,e}$,
  \begin{align*}
    (u \otimes v) \cdot (T_{1}\circ T'_{1})(x\otimesB y) &=
    \sum ux_{(1)} \otimesB vx_{(2)} S(S^{-1}(y)x_{(3)}) \\
    &= \sum ux_{(1)} \otimesB vx_{(2)}S(x_{(3)})y \\
    &= \sum ux_{(1)} \otimesB vr(\epsilon^{\sharp}(x_{(2)}))
    y = \sum us( \epsilon^{\sharp}(x_{(2)}))x_{(1)} \otimesB
    vy =
    ux\otimesB vy, \\
    (u \otimes v) \cdot (T'_{1}\circ T_{1})(x\otimesB y) &=
    \sum ux_{(1)} \otimesB vS(S^{-1}(x_{(3)}y)x_{(2)}) \\
    &=
    \sum ux_{(1)} \otimesB vS(x_{(2)})x_{(3)}y \\
    &=\sum ux_{(1)} \otimesB vs(\epsilon^{\flat}(x_{(2)}))y =
    \sum ux_{(1)}s(\epsilon^{\flat}(x_{(2)})) \otimesB vy = ux
    \otimesB vy. \qedhere
  \end{align*}
\end{proof}
As in the case of multiplier bialgebras or Hopf algebroids,
this result should have a converse.

\subsection{Bi-measured multiplier
  $(B,\Gamma)$-$*$-bialgebroids} \label{subsection:bimeasured}
We now introduce the main objects of this article ---
multiplier $(B,\Gamma)$-Hopf $*$-algebroids equipped with
certain integrals. In \S\ref{section:reduced}, we shall
construct completions of such objects in the form of
measured quantum groupoids. 

As on a groupoid, integration on a multiplier
$(B,\Gamma)$-$*$-bialgebroid $(A,\Delta)$
proceeds in stages. First, one needs partial integrals
$\phi,\psi \colon A \to B$ with suitable left or right
invariance properties, and second a suitable weight $\mu
\colon B\to \complex$ that is compatible with the action of
$\Gamma$.  The results in \cite{koelink:su2} suggest that
dynamical quantum groups that are compact in a suitable
sense even possess a bi-invariant integral $h\colon A\to
B\otimes B$ that can be obtained from a Peter-Weyl
decomposition of $A$. 

We first focus on the weight $\mu$ and the bi-integral $h$,
and discuss left and right integrals in the next subsection.

Let us briefly recall some terminology.  Let $C$ be a
 $*$-algebra with local units. A linear map $\mu\colon C\to
\complex$ is {\em faithful} if $\mu(Cc)=0$
implies $c=0$, and {\em positive} if
$\mu(c^{*}c)\geq 0$ for all $c\in C$. Assume that $\mu$ is
positive.  Then $\mu$ is $*$-linear, because positivity of
$\phi((b+c)^{*}(b+c))$ and $\phi((b+ic)^{*}(b+ic))$ implies $\mu(b^{*}c)=\overline{\phi(c^{*}b)}$ for all
$b,c\in C$, and faithful
as soon as $\mu(c^{*}c) \neq 0$ whenever $c\neq 0$.
\begin{definition} \label{definition:modular-qinvariant} A
  \emph{weight} for $(B,\Gamma)$ is a faithful, positive
  linear map $\mu \colon B \to \complex$ that is {\em
    quasi-invariant with respect to $\Gamma$} in the sense
  that for each $\gamma \in \Gamma$, there exists some
  $D_{\gamma} \in M(B)$ such that $\mu(\gamma(b D_{\gamma}))
  = \mu(b)$ for all $b \in B$.
\end{definition}
\begin{remark}
  \label{remark:modular-cocycle} Let $\mu$ be a weight
  for $(B,\Gamma)$. Then
  \begin{enumerate}
  \item each $D_{\gamma}$ is uniquely determined and
    self-adjoint,
\item $ D_{\gamma\gamma'} = \gamma'{}^{-1}(D_{\gamma})
  D_{\gamma'}$ and  $
  1=\gamma^{-1}(D_{\gamma^{-1}})D_{\gamma}$ for all
  $\gamma,\gamma' \in \Gamma$,
\item $\mu(\gamma^{-1}(b)c) =
  \mu(b\gamma(c)D_{\gamma^{-1}}^{-1}))) =
  \mu(b\gamma(cD_{\gamma}))$ for all $b,c\in B,\gamma \in
  \Gamma$.
\end{enumerate}
Indeed, i) and ii) follow easily from the fact that $\mu$ is
faithful and the relations $\mu(\gamma(bD_{\gamma}^{*})) =
\overline{\mu(\gamma(D_{\gamma}b^{*}))} =
\overline{\mu(b^{*})} = \mu(b)$ and
$\mu(\gamma(\gamma'(bD_{\gamma\gamma'}))) = \mu(b) =
\mu(\gamma'(bD_{\gamma'})) =
\mu(\gamma(\gamma'(bD_{\gamma'})D_{\gamma}))$.    
\end{remark}

  Let
$(A,\Delta)$ be a multiplier $(B,\Gamma)$-$*$-bialgebroid.

The following definition is inspired by the notion of a Haar
functional introduced in \cite{koelink:su2}.
\begin{definition}\label{definition:bi-integral}
  A \emph{bi-integral} on  $(A,\Delta)$ is a morphism of
  $(B,\Gamma)^{\ev}$-modules $h\colon A \to B\otimes B$
  satisfying $\Delta(\ker h)(1 \todot A_{e,e}) \subseteq
  \ker h \todot A$ and $\Delta(\ker h)(A_{e,e} \todot 1)
  \subseteq A \todot \ker h$.  If $(A,\Delta)$ is proper and
  $h(r(b)s(b'))=b\otimes b'$ for all $b,b'\in B$, we call
  such a bi-integral {\em normalized}.
\end{definition}
\begin{lemma} \label{lemma:bi-integral} Let $(A,\Delta)$ be
  proper and let $h$ be a normalized bi-integral on $(A,\Delta)$.
  \begin{enumerate}
  \item $ (\Id \todot m_{B}\circ h) \circ \Delta = h =
    (m_{B}\circ h\todot \Id) \circ \Delta$, where
    $m_{B}\colon B\otimes B\to B$ denotes the
    multiplication.
  \item If $h'$ is a normalized bi-integral on $(A,\Delta)$,
    then $h'=h$.
  \item If $(A,\Delta,\epsilon,S)$ is a proper multiplier
    $(B,\Gamma)$-Hopf $*$-algebroid, then $h \circ S =\sigma_{B}
    \circ h$, where $\sigma_{B} \colon B\otimes B \to
      B\otimes B$ denotes the flip $b\otimes c\mapsto
      c\otimes b$.
  \end{enumerate}
\end{lemma}
\begin{proof}
  i) We only prove the first equation.  Let $\omega\colon
  (A,r) \to B$ be a morphism of $(B,\Gamma)$-modules sending
  $I:=\ker h$ to $0$. Then
  \begin{align*} 
    (\Id \todot \omega)(\Delta(I))A_{e,e} =
    (\Id \otimesB \omega)(\Delta(I)(A_{e,e} \todot 1))
    \subseteq     (\Id \otimesB \omega)(A \otimesB I)  = 0
  \end{align*}
  and hence $(\Id \todot \omega)(\Delta(I)) = 0$.
  Moreover, if $b,b',b'' \in B$ and $u\in
  A_{e,e}$, then
  \begin{align*}
    (\Id \todot \omega)(\Delta(r(b)s(b')))s(b'')u &= (\Id
    \otimesB \omega)(r(b)s(b'')u \otimesB s(b'))
 =
    r(b) s(\omega(s(b')r(b'')))u.
  \end{align*}
  For $\omega=m_{B} \circ h$, these calculations imply
  for all $a\in I$ and $b,b' \in B$  \begin{align*}
    (\Id \todot m_{B}\circ h)(\Delta(a)) &= 0 = h(a), & (\Id
    \todot m_{B}\circ h)(\Delta(r(b)s(b'))) &=r(b)s(b')=
    h(r(b)s(b')).
  \end{align*}
  Since $A = I + r(B)s(B)$, we can conclude $ (\Id \todot
  m_{B}\circ h) \circ \Delta = h$.

    ii) Let $x \in \ker h$ and choose $u,u'\in B\otimes B$
    such that $u(1\otimes m_{B}(u'))h'(x)=h'(x)$. Then
    \begin{align*}
      h'(x) = h(u
      h'(x)s(m_{B}(u')))  =
      \sum h(u x_{(1)}s(m_{B}(h'(x_{(2)}u')))) = 0
    \end{align*}
    because $\sum ux_{(1)} \otimes x_{(2)}u' \in u(\ker h)
    \otimesB A$. Thus, $\ker h \subseteq \ker h'$. Since $h$
    and $h'$ are normalized and $\ker h + B\otimes B = A$,
    we can conclude $h=h'$.
 
    iii) One easily verifies that $\sigma_{B} \circ h \circ
    S$ is a normalized bi-integral. By ii), it equals $h$.
\end{proof}

\begin{definition} \label{definition:bi-measured}
  A proper multiplier $(B,\Gamma)$-$*$-bialgebroid
  $(A,\Delta)$ is \emph{bi-measured} if it is
  equipped with a normalized bi-integral $h\colon A\to
  B\otimes B$ and a weight $\mu$ for $(B,\Gamma)$ such that
  $\nu:=(\mu \otimes \mu)\circ h$ is faithful and positive.  
\end{definition}
\begin{remark} \label{remark:bi-measured}
  Given a bi-measured proper multiplier $(B,\Gamma)$-Hopf
  $*$-algebroid as above, $h$ is evidently faithful, and
  also $*$-linear.  To see this, note that $(\mu
  \otimes \mu)(h(a^{*})(b\otimes c)) = \nu(a^{*}r(b)s(c)) =
  \overline{\nu(s(c^{*})r(b^{*})a)} = \overline{(\mu \otimes
    \mu)((b\otimes c)^{*}h(a))} = (\mu\otimes
  \mu)(h(a)^{*}(b\otimes c))$ for all $a\in A,b,c\in B$.
\end{remark}
\subsection{Left and right
  integrals} \label{subsection:integrals} For large parts of
this article, the multiplier $(B,\Gamma)$-Hopf
$*$-algebroids under consideration need not be equipped with
a bi-integral, but only with left and right integrals
$\phi,\psi$.  The definition of these integrals involves
slice maps of the following form.  

Let $(A,\Delta)$ be a multiplier
$(B,\Gamma)$-$*$-bialgebroid and let $\phi \colon (A,r) \to
B$ be a morphism of $(B,\Gamma)$-modules. Then there exists
a unique linear map $\Id \todot \phi\colon \tilde M(A \todot
A) \to M(A)$ such that
\begin{align*}
  ((\Id \todot \phi)(T))a &= (\Id \otimesB \phi)(T(a
  \otimes 1)), & a((\Id \todot \phi)(T)) &= (\Id \otimesB
  \phi)((a \otimes 1)T)
\end{align*}
for all $T\in \tilde M(A\todot A)$ and $a\in A$, where we
regard $T(a \otimes 1)$ and $(a \otimes 1)T$ as elements of
$\sA \otimesB \rA$ and $\As \otimesB \Ar$, respectively.  In
the case $T=\Delta(x)$ for some $x \in A$,
\begin{align} \label{eq:slice-delta}
  (\Id \todot \phi)(\Delta(x))a &= \sum
  s(\phi(x_{(2)}))x_{(1)}a,  &
a  (\Id \todot \phi)(\Delta(x)) &= \sum
ax_{(1)} s(\phi(x_{(2)})).
\end{align}
Likewise, every morphism $\psi \colon (A,s) \to B$ yields a
slice map $\psi \todot \Id \colon \tilde M(A\todot A) \to
M(A)$.
\begin{definition} \label{definition:algebra-integrals} A
  \emph{left integral} on $(A,\Delta)$ is a morphism
  $\phi\colon (A,r) \to B$ satisfying $(\Id \todot \phi)
  \circ \Delta = r\circ \phi$. A \emph{right integral}
  on $(A,\Delta)$ is a morphism $\psi\colon (A,s) \to B$
  satisfying $(\psi \todot \Id) \circ \Delta = s\circ \psi$.
\end{definition}
\begin{remarks} \label{remarks:algebra-integrals}
  \begin{enumerate}
  \item In Sweedler notation, the invariance conditions on
    $\phi$ and $\psi$ become
    \begin{align*}
     \sum s(\phi(x_{(2)}))x_{(1)}a &= r(\phi(x))a, &
     \sum ax_{(2)}r(\psi(x_{(1)})) &= as(\psi(x))
   \quad \text{for all } a,x\in A.
    \end{align*}
  \item If $(A,\Delta,\epsilon,S)$ is a $(B,\Gamma)$-Hopf
    $*$-algebroid, then the map $\phi \mapsto \phi \circ S$
    gives a bijection between left and right integrals on
    $(A,\Delta)$. This follows easily from
    \eqref{eq:antipode-sigma}.
  \item If $\phi$ is a left integral, then also
    $\phi(-s(b))$ is left integral for each $b\in
    B$. Likewise, if $\psi$ is a right integral, then also
    $\psi(-r(b))$ is a right integral for each $b\in B$.
  \end{enumerate}
\end{remarks}
We shall frequently use the following strong invariance
relations:
\begin{proposition} \label{proposition:integral-strong-invariance}
Assume that $(A,\Delta,\epsilon,S)$ is a $(B,\Gamma)$-Hopf
$*$-algebroid.  Then
  \begin{enumerate}
  \item  $(\Id \otimesB \phi)((1\todot z)\Delta(x)) =S((\Id
  \otimesB \phi)(\Delta(z)(1\todot x)))$ for every left
  integral $\phi$ and all $x,z\in A$;
\item $(\psi \otimesB \Id)(\Delta(x)(z\todot 1)) = S((\psi
  \otimesB \Id)((x\todot
  1)\Delta(z))$ for every right
  integral $\psi$ and all $x,z\in A$.
  \end{enumerate}
\end{proposition}
\begin{proof}
Using Sweedler notation, we calculate
  \begin{align*}
    \sum x_{(1)}s(\phi(zx_{(2)})) &= \sum
    x_{(1)}s(\phi(z_{(2)}r(\epsilon^{\flat}(z_{(1)}))x_{(2)})) \\
    & = \sum
    s(\epsilon^{\flat}(z_{(1)}))x_{(1)}s(\phi(z_{(2)}x_{(2)}))
    \\
    &= \sum S(z_{(1)})z_{(2)}x_{(1)}s(\phi(z_{(3)}x_{(2)}))
    = \sum S(z_{(1)})r(\phi(z_{(2)}x))
\end{align*}
and
\begin{align*}
    \sum r(\psi(x_{(1)}z)x_{(2)} &= \sum
    r(\psi(x_{(1)}s(\epsilon^{\sharp}(z_{(2)}))z_{(1)})) x_{(2)} \\
    &= \sum r(\psi(x_{(1)}z_{(1)}))x_{(2)}r(\epsilon^{\sharp}(z_{(2)})) \\
    &= \sum r(\psi(x_{(1)}z_{(1)}) x_{(2)}z_{(2)}S(z_{(3)})
    = s(\psi(xz_{(1)}))S(z_{(2)}). \qedhere
\end{align*}
\end{proof}

Normalized bi-integrals yield left and right integrals as follows:
\begin{lemma} \label{lemma:integral-biintegral} Assume that
  $(A,\Delta)$ is proper, $h$ is a normalized bi-integral on
  $(A,\Delta)$, and $\mu \colon B\to \complex$ is
  linear. Then $\phi:=(\Id \otimes \mu) \circ h$ and $\psi
  :=(\mu \otimes \Id) \circ h$ are a left and a right
  integral, respectively, and $\phi\circ S^{\pm 1} = \psi$.
\end{lemma}
\begin{proof}
Repeating the proof of Lemma \ref{lemma:bi-integral} i) with
$\omega:=\phi=(\Id \otimes \mu)\circ h$, we find
\begin{align*}
    (\Id \todot \phi)(\Delta(a)) &= 0 = r(\phi(a)), & (\Id
    \todot \phi)(\Delta(r(b)s(b'))) &=r(b\mu(b')) =
    \phi(r(b)s(b'))
  \end{align*}
  for all $a\in \ker h$ and $b,b' \in B$.  Since $A = (\ker h) +
  r(B)s(B)$, we can conclude $(\Id \todot \phi) \circ \Delta
  = r\circ \phi$. The assertion on $\psi$ follows
  similarly, and the last equation follows from  Lemma
  \ref{lemma:bi-integral} iii).
\end{proof}

\subsection{Measured multiplier
  $(B,\Gamma)$-$*$-bialgebroids} \label{subsection:measured} 

Much of the ensuing material applies not only to
bi-measured proper multiplier $(B,\Gamma)$-Hopf
$*$-algebroids but also to  the following class of objects.
\begin{definition} \label{definition:measured}
  A a multiplier $(B,\Gamma)$-$*$-bialgebroid
  $(A,\Delta)$ is \emph{measured} if it is
  equipped with a left integral $\phi$, a right integral
  $\psi$, and a weight $\mu$ for $(B,\Gamma)$ such that
  $\nu:=\mu\circ \phi$ and $\nu^{-1}:=\mu \circ \psi$ are
  faithful, positive, and coincide, and $\psi(A)=B=\phi(A)$.
\end{definition}
\begin{remarks} \label{remarks:measured}
  \begin{enumerate}
  \item Given a measured multiplier $(B,\Gamma)$-Hopf
    $*$-algebroid as above, the maps $\phi$ and $\psi$ are
    $*$-linear. This can be seen from a similar argument as in  Remark \ref{remark:bi-measured}.
  \item If $(A,\Delta,\epsilon,S,h,\mu)$ is a bi-measured
    proper multiplier $(B,\Gamma)$-Hopf $*$-algebroid and
    $\phi=(\mu \otimes \Id)\circ h$ and $\psi=(\Id \otimes
    \mu) \circ h$, then
    $(A,\Delta,\epsilon,S,\phi,\psi,\mu)$ is a measured
    multiplier $(B,\Gamma)$-Hopf $*$-algebroid by Lemma
    \ref{lemma:integral-biintegral}. In that case,  $\phi
    \circ S^{\pm 1} = \psi$ and $\nu \circ S = \nu$  by Lemma
    \ref{lemma:bi-integral} iii).
  \item One could probably drop the assumption
    $\nu=\nu^{-1}$ and assume the existence of an invertible
    multiplier $\delta$ such that $\nu^{-1}(a)=\nu(a\delta)$
    for all $a\in A$. In the applications we have in mind,
    in particular, in the bi-measured case, the stricter
    assumption above is satisfied.
  \end{enumerate}
\end{remarks}

Till the end of this subsection, let
$(A,\Delta,\epsilon,S,\phi,\psi,\mu)$ be a measured
multiplier $(B,\Gamma)$-Hopf $*$-algebroid.
Define $D,\bar D\colon A\to A$ by
\begin{align} \label{eq:modular-d}
D(a) &= r(D_{\partial_{a}^{-1}})a
    =
    ar(D_{\partial_{a}}^{-1}),   &
    \bar D(a) &=s(
    D_{\bar \partial_{a}^{-1}})a = as(
    D_{\bar \partial_{a}}^{-1}) &&\text{for all } a\in A.
\end{align}
\begin{lemma} \label{lemma:modular-d} $D$ and $\bar D$ both
  are algebra and $(B,\Gamma)^{\ev}$-module automorphisms of
  $A$, and satisfy
  \begin{gather*}
    \begin{aligned}
      (D \todot \Id) \circ \Delta &= \Delta \circ D, & (\Id
      \todot \bar D) \circ \Delta &= \Delta \circ \bar D, &
      (\bar D \todot \Id) \circ \Delta &= (\Id \todot D)
      \circ \Delta,
    \end{aligned} \\
    \begin{aligned}
D\circ \bar D &= \bar D \circ D, &      S \circ D &= \bar D^{-1} \circ S, & S \circ \bar D &=
      D^{-1} \circ S, & * \circ D &= D^{-1} \circ *, &
      *\circ \bar D &= \bar D^{-1} \circ *.
    \end{aligned}
\end{gather*}
\end{lemma}
\begin{proof}
  The maps $D$ and $\bar D$ are bijective because
  $D_{\gamma}$ is invertible for each $\gamma \in
  \Gamma$.  The remaining assertions follow from
  straightforward calculations, for example,
 \begin{gather*}
     D(xy) = r(D_{\partial_{xy}^{-1}})xy =
     r(D_{\partial_{x}^{-1}}\partial_{x}(D_{\partial_{y}^{-1}}))xy
     =r (D_{\partial_{x}^{-1}})xr(D_{\partial_{y}^{-1}})y
     =D(x)D(y), \\ 
     S(D(x)) = S(r(D_{\partial_{x}}^{-1})x) =
     S(x)s(D_{\bar \partial_{S(x)}}) = \bar
     D^{-1}(S(x)), \\ 
  D(x)^{*}=x^{*}r(D^{*}_{\partial_{x}^{-1}}) = x^{*}
   r(D_{\partial_{x^{*}}}) = D^{-1}(x^{*}) \quad \text{for
     all } x,y\in A. \qedhere
  \end{gather*}
\end{proof}
\begin{lemma} \label{lemma:modular-bimodule} Let $\omega \in
  \{\phi,\psi,\nu\}$.
  \begin{enumerate}
  \item $\omega (A_{\gamma,\gamma'}) = 0$ whenever
    $(\gamma,\gamma') \neq (e,e)$.
  \item $\omega(r(b)s(b')a)=\omega(ar(b)s(b'))$ for all $a\in A$,
    $b,b' \in B$.
  \item $\omega(D(a)a') = \omega(aD^{-1}(a'))$ and
    $\omega(\bar D(a)a') = \omega(a\bar D^{-1}(a'))$ for all
    $a,a'\in A$.
  \end{enumerate}
\end{lemma}
\begin{proof}
  i) For $\omega=\nu$, the assertion follows from the
  relation $\ker \phi + \ker \psi \subseteq \ker \nu$.  To
  obtain the assertion for $\omega=\phi,\psi$, use the fact
  that $\mu$ is faithful.

  ii) Let $a \in A$ and $b,b' \in B$. Then $\nu(r(b)a) =
  \mu(b\phi(a)) = \mu(\phi(a)b) = \nu(ar(b))$ and similarly
  $\nu(s(b')a)=\nu(as(b'))$.  To obtain the assertion for
  $\omega=\phi,\psi$, use the fact that $\mu$ is faithful
  again.

  iii) This follows immediately equation
  \eqref{eq:modular-d} and i).
\end{proof}

We shall now show that $\nu=\mu\circ \phi$ has a modular
automorphism and thus satisfies an algebraic variant of the
KMS-condition. Let us briefly recall this concept.

Let $C$ be a $*$-algebra with local units and a faithful,
positive, linear map $\omega \colon C\to \complex$.  A
\emph{modular automorphism} for $\omega$ is a bijection
$\theta_{\omega} \colon C\to C$ satisfying $\omega(cc') =
\omega(c'\theta_{\omega}(c))$ for all $c,c' \in C$.  If it
exists, a modular automorphism $\theta_{\omega}$ for
$\omega$ is uniquely determined, an algebra automorphism,
and satisfies $\omega\circ \theta_{\omega} =\omega$ and
$\theta_{\omega} \circ \ast \circ \theta_{\omega} \circ \ast
= \Id_{C}$. This follows easily from the relations
\begin{align*}
  \omega(z\theta_{\omega}(xy))&=\omega(xyz)
  =\omega(yz\theta_{\omega}(x))
  =\omega(z\theta_{\omega}(x)\theta_{\omega}(y)), \\ 
  \omega(yx) &= \overline{\omega(x^{*}y^{*})} =
  \overline{\omega(y^{*}\theta_{\omega}(x^{*}))} =
  \omega(\theta_{\omega}(x^{*})^{*}y) =
  \omega(y\theta_{\omega}(\theta_{\omega}(x^{*})^{*})),
\end{align*}
where $x,y,z\in C$.

As before, let $(A,\Delta,\epsilon,S,\phi,\psi,\mu)$ be a measured
multiplier $(B,\Gamma)$-Hopf $*$-algebroid.
\begin{theorem} \label{theorem:modular}
  \begin{enumerate}
  \item There exists a modular automorphism $\theta$ for
    $\nu$.
  \item $\theta$ is a $(B,\Gamma)^{\ev}$-module
    automorphism of $A$.
  \item If $\nu \circ S = \nu$, then
    $\theta \circ S = S \circ \theta^{-1}$.
  \end{enumerate}
\end{theorem}
\begin{proof}
i) The proof repeatedly uses  strong invariance of $\phi$ and
$\psi$, and closely follows \cite{daele}, where the
corresponding result was obtained for multiplier Hopf
algebras. We proceed in three steps. 

  {\em Step 1. }  Repeatedly using Remark
  \ref{remark:modular-cocycle} iii), we
  find that for all $x,x',y,y' \in A$,
  \begin{align} 
    \bar \partial_{x'} &= \bar \partial_{y'}^{-1} &
    \Rightarrow&& \nu^{-1}(ys(\psi(xx'))y') &= \mu(\psi(yy')
    \bar \partial_{y'}^{-1}(\psi(xx'))) \nonumber  \\ 
    &&&& &= \mu(\psi(xx')\bar \partial_{y'}(\psi(yy')
    D_{\bar \partial_{y'}})) = \nu^{-1}(xs(\psi(y y'))\bar
    D(x')),   \label{eq:modular-1} 
\\ \nonumber
    \bar \partial_{x} &=\partial_{y'} &
    \Rightarrow&&
    \nu(yr(\psi(xx'))y') &=
    \mu(\phi(yy') \partial_{y'}^{-1}(\psi(xx'))) \\
    &&&&      &= \mu(\partial_{y'}(\phi(yy')D_{\partial
      y'})\psi(xx'))  =\nu^{-1}(xs(
    \phi(D(y)y'))x'),  \label{eq:modular-2} 
\\ \nonumber
    \partial_{x} &= \partial_{y}^{-1} &\Rightarrow &&
    \nu(yr(\phi(xx'))y') &=
    \mu(\partial_{y}(\phi(xx'))\phi(yy') ) \\
    &&&& &=
    \mu(\partial_{y}^{-1}(\phi(yy')D_{\partial_{y}^{-1}})\phi(xx'))
    = \nu(x r(\phi(D(y)y'))x').  \label{eq:modular-3}
  \end{align}
  
{\em Step 2. } Let $c,d \in A$ and
  \begin{align} \label{eq:modular-4}
 a &= \sum \bar D(s(\psi(d
  S(c_{(2)})))c_{(1)}) \in A, & a' &= \sum d_{(2)}r(\phi(D(S(
  d_{(1)}))\bar D(c))) \in A.
  \end{align}
Then the equations above and Proposition
\ref{proposition:integral-strong-invariance} imply 
  \begin{align*}
    \nu(za) &=  \sum \nu(z\bar D(s(
    \psi(d S(c_{(2)})))c_{(1)})) &&   \\
    &= \sum \nu(ds( \psi( zc_{(1)}))S(c_{(2)})) & &
    \text{(Equation \eqref{eq:modular-1})}
    \\
    &= \sum \nu( dr(\psi(z_{(1)}c))z_{(2)}) &&
    \text{(Proposition
      \ref{proposition:integral-strong-invariance}) }
    \\
    &= \sum \nu(z_{(1)}s( \phi( D( d)z_{(2)}))c) &&
    \text{(Equation \eqref{eq:modular-2})}
    \\
    &= \sum \nu(S(D(d_{(1)}))r(\phi( d_{(2)}z)) c) &&
    \text{(Proposition
      \ref{proposition:integral-strong-invariance}) }
    \\
    &= \sum \nu( S(d_{(1)})r(\phi(
    d_{(2)}z)) \bar D(c)) 
&& \text{(use $S\circ D = \bar D^{-1} \circ S$ and
    \ref{lemma:modular-bimodule} iii))}\\
    &= \sum \nu( d_{(2)}r(\phi(D( S( d_{(1)}))\bar D(c)))z)=
    \nu(a'z). && \text{(Equation \ref{eq:modular-3})}
  \end{align*}
  
  {\em Step 3. } Using bijectivity of the maps $\bar D,S,
  T_{1}$ and the relation $\langle s(\psi(A))A\rangle=A$,
  one finds that all elements of the form like $a$ in
  \eqref{eq:modular-4} span $A$. A similar argument shows
  that the same is true for elements of the form like
  $a'$. Hence, there exists a bijection $\theta \colon A \to
  A$ such that $\nu(az)=\nu(z\theta(a))$ for all $a \in A$,
  and uniqueness of such a bijection follows from
  faithfulness of $\nu$.

  ii) We first show that $\theta$ respects the
  grading. Let $c,d\in A$ be homogeneous. Then the element
  $a$ in \eqref{eq:modular-4} is homogeneous as well, with
  grading given by $\partial_{a}=\partial_{c}$ and
  $\bar \partial_{a} = \bar \partial_{d}$ because
  $\psi(dS(c_{(2)}))=0$ unless
  $\bar \partial_{d}=\partial_{c_{(2)}}=\bar\partial_{c_{(1)}}$,
  and similarly $a'$ in \eqref{eq:modular-4} is homogeneous
  with the same degree like $a$. To see that $\theta$ is
  $B\otimes B$-linear, use the relation
  $\nu(y\theta(r(b)s(b')x)) =
  \nu(r(b)s(b')xy)=\nu(xyr(b)s(b'))=\nu(yr(b)s(b')\theta(x))$,
where $x,y\in A$ and $b,b'\in B$, and faithfulness
  of $\nu$.

  iii) If $\nu \circ S = \nu$, then we have $\nu(y\theta(S(x))) =
  \nu(S(x)y) = \nu(S^{-1}(y)x) =
  \nu(\theta^{-1}(x)S^{-1}(y)) = \nu(yS(\theta^{-1}(x)))$
  for all $x,y \in A$.
\end{proof}
 Define $\theta_{D},\theta_{\bar D},\theta_{D,\bar
  D}\colon A\to A$ by  
\begin{align*}
  \theta_{D}&:=\theta \circ D^{-1} = D^{-1} \circ \theta, &
  \theta_{\bar D}&:=\theta \circ \bar D^{-1} =\bar D^{-1}
  \circ \theta, &
\theta_{D,\bar D} := \theta \circ D^{-1} \circ \bar D^{-1}.
\end{align*}
\begin{proposition} \label{proposition:modular-integrals}
  \begin{enumerate}
  \item $\phi \circ \theta = \phi$ and
    $\phi(xy)=\partial_{x}(\phi(y\theta_{D}(x)))$ for all
    $x,y \in A$.
  \item $\psi \circ \theta =
    \psi$ and $\psi(xy) =\bar \partial_{x}(\psi(y\theta_{\bar
      D}(x)))$ for all $x,y \in A$.
  \item $h \circ \theta = h$ and $h(xy) = (\partial_{x}
    \otimes \bar \partial_{x})(h(y\theta_{D,\bar D}(x)))$ for
    all $x,y \in A$ if $h$ is a bi-invariant integral and
    $\nu = (\mu \otimes \mu) \circ h$.
 \end{enumerate}
\end{proposition}
\begin{proof}
 Assertion i) follows from the fact that $\mu$
  is faithful and that for all $x,y\in A$, $b\in B$,
\begin{align*}
  \mu(b\phi(\theta(x))) = \nu(r(b)\theta(x))
  &=\nu(\theta(r(b)x)) = \nu(r(b)x) = \mu(b\phi(x)), \\
  \mu(b \phi(y\theta(x))) = \nu(r(b)y\theta(x))&=
  \nu(xr(b)y) \\
  &=\nu(r(\partial_{x}(bD_{\partial_{x}}))xr(D_{\partial_{x}}^{-1})y)
  \\
  &= \mu(\partial_{x}(bD_{\partial_{x}})\phi(D(x)y)) 
  = \mu(b \partial_{x}^{-1}(\phi(D(x)y))).
  \end{align*}
 Assertions ii) and iii) follow similarly.
\end{proof}
  \begin{proposition} \label{proposition:modular-delta}
Assume that $\As$    is a flat $B$-module. Then $\Delta \circ \theta_{D} = (S^{2} \todot \theta_{D})
    \circ \Delta$.
\end{proposition}
\begin{proof}
  Let $x,y \in A$. Using Sweedler notation, we calculate
  \begin{align*}
    \sum
    \theta_{D}(x)_{(1)}s(\phi(y\theta_{D}(x)_{(2)}))
    &= \sum S(s(\phi(y_{(2)}\theta_{D}(x)))y_{(1)}) &&
    \text{(Proposition
      \ref{proposition:integral-strong-invariance})}
    ) \\
    &= \sum S(s(\partial_{x}^{-1}(\phi(xy_{(2)})))y_{(1)})
    && \text{(Proposition
      \ref{proposition:modular-integrals})} \\
    &= \sum S(y_{(1)}s(\phi(xy_{(2)})))
    &&  \\
    &= \sum S^{2}(s(\phi(x_{(2)}y))x_{(1)}) &&
    \text{(Proposition
      \ref{proposition:integral-strong-invariance})} \\
    &= \sum S^{2}(s(\partial_{x_{(2)}}
    (\phi(y\theta_{D}(x_{(2)}))))x_{(1)}) &&
    \text{(Proposition
      \eqref{proposition:modular-integrals})}\\
    &= \sum S^{2}(x_{(1)})s(\phi(y\theta_{D}(x_{(2)}))).
  \end{align*}
  Since $\As$ is a flat $B$-module and maps of the form $a
  \mapsto \phi(ya)$, where $y \in A$, separate the points of
  $A$, we can conclude $\sum \theta_{D}(x)_{(1)} \todot
  \theta_{D}(x)_{(2)} = \sum S^{2}(x_{(1)}) \todot
  \theta_{D}(x_{(2)})$.
\end{proof}

\subsection{The dual $*$-algebra} \label{subsection:dual}
Let $(A,\Delta,\epsilon,S,\phi,\psi,\mu)$ be a measured
multiplier $(B,\Gamma)$-Hopf $*$-algebroid.  Denote by
$M(A)'$ the dual vector space of $M(A)$ and let
\begin{align*}
  \hA:= \{ \nu(x-) : x\in A\} \subseteq M(A)'
\end{align*}
Then $\hA = \{\nu(-x) :x\in A\}$ by Theorem
\ref{theorem:modular} and for each $\omega \in \hA$, there
exist unique $B$-module maps $_{r}\omega \colon {_{r}M(A)}
\to B$, $\omega_{r} \colon M(A)_{r} \to B$, $_{s}\omega
\colon {_{s}M(A)} \to B$, $\omega_{s} \colon M(A)_{s} \to B$
whose compositions with $\mu$ are equal to $\omega$, because
$\nu=\mu\circ \phi = \mu \circ \psi$ and $\mu$ is faithful.
Using either of these $B$-module maps, one can equip $\hA$
with the structure of a $*$-algebra. We shall choose an
approach that fits well with the duality on the
operator-algebraic level in the next section.

First, we define an abstract Fourier transform
\begin{align*}
  A \to \hat A, \quad x\mapsto \hat x:=\nu(S(x)-).
\end{align*}
Evidently, $ \hat x_{s} = \psi(S(x)-)$ and $\hat x_{r} =
\phi(S(x)-)$, and by Proposition
\ref{proposition:modular-integrals}, $ _{s}\hat{x} =
\psi(-\theta(S(x)))$ and $_{r}\hat{x} =
\phi(-\theta(S(x)))$.  For all $x,a\in A$, we define  a right
convolution 
\begin{align} \label{eq:convolution-1} a \ast\hat x & :=
  \sum a_{(2)}r(\hat x_{s}(a_{(1)})) = \sum a_{(2)}
  r(\psi(S(x)a_{(1)})) \in A.
\end{align}
\begin{remark}
  One could also work with the transform $A \to \hat A$, $x
  \mapsto \check x := \nu(-S(x))$, and the left convolution
  defined by
  \begin{align} \label{eq:convolution-left}
    \check x \ast a &:=  \sum s(_{r}\check
    x(a_{(2)}))a_{(1)} = \sum s(\phi(a_{(2)}S(x)))a_{(1)}
    \in A
    \quad \text{for all } x,a\in A.
  \end{align}
  If $\phi\circ S=\psi$, for
  example, if we are in the bi-measured case (see Remark
  \ref{remarks:measured} ii)),  then
  \begin{align*}
    \widecheck{S(x)}  \ast S(a) = \sum
    s(\phi(S(a)_{(2)}S^{2}(x)))S(a)_{(1)} =
    S(a_{(2)}r(\psi(S(x)a_{(1)}))) = S(a \ast \hat x) \quad
    \text{for all } a,x\in A.
  \end{align*}
\end{remark}
We collect a few useful formulas. First, for all $a,x \in
A$, 
\begin{align} \label{eq:convolution-alt}
a
  \ast\hat x &=  \sum r(\psi(a_{(1)}\theta_{\bar
    D}(x)))a_{(2)},  && (\text{Proposition
    \ref{proposition:modular-integrals}}) \\  
a \ast \hat x &= \sum S^{-1}(r(\psi(S(x)_{(1)}a))S(x)_{(2)})
= \sum x_{(1)}s(\psi(S(x_{(2)})a))
&& (\text{Proposition
  \ref{proposition:integral-strong-invariance}}) \label{eq:convolution-alt-2}
\end{align}
Next, for all $a,x,y\in A$, $b\in B$,
$\gamma,\gamma',\delta,\delta' \in \Gamma$,
\begin{gather}
 \label{eq:convolution-module} 
  \begin{aligned}
    r(b)a \ast \hat{ x} &= a\ast \widehat{s(b)x}, & ar(b)
    \ast \hat{x} &= a\ast \widehat{xs(b)}, \\ s(b)a \ast
    \hat{x} &= s(b)(a\ast \hat x), & as(b) \ast \hat x &=
    (a\ast\hat x)s(b),
  \end{aligned} \\
 \label{eq:convolution-product}
  \begin{aligned}
    (a \ast \hat x) \ast \hat y &= \sum 
    a_{(3)}r(\psi(S(y)a_{(2)}r(\psi(S(x)a_{(1)}))) \\
&= \sum a_{(2)} r(\psi(S(y)x_{(1)}s(\psi(S(x_{(2)})a_{(1)}))))\\
&= \sum a_{(2)}r(\psi(S(x_{(2)}r(\psi(S(y)x_{(1)})))a)) 
    = a \ast \widehat{(x \ast \hat{\!\!\!y\,\,\,})},
  \end{aligned} \\
 \label{eq:convolution-grading}
  A_{\gamma,\gamma'} \ast
    \widehat{A_{\delta,\delta'}}  \subseteq
    \sum_{\gamma''}
    s(\psi(A_{\delta'{}^{-1},\delta^{-1}}A_{\gamma,\gamma''}))A_{\gamma'',\gamma'}
    \subseteq \delta_{\gamma,\delta'} A_{\delta,\gamma'},
\end{gather}
where we used Lemma
 \ref{lemma:modular-bimodule} in the last line.

The $(B,\Gamma)^{\ev}$-algebra structure on $A$
induces the following structure on $\hat A$:
\begin{definition}
  A {\em $(B,\Gamma)^{\ev}$-matrix-algebra} is a
  non-degenerate $*$-algebra $\hA$ equipped with a
  non-degenerate $*$-homomorphism $B\otimes B \to M(\hA)$
  and a direct sum decomposition $\hA =
  \bigoplus_{\gamma,\gamma' \in \Gamma}
  \hA^{\gamma,\gamma'}$ as a vector space such that
  \begin{gather*}
    \begin{aligned}
      \hA^{\gamma,\gamma'} \hA^{\delta,\delta'} &\subseteq
      \delta_{\gamma',\delta} \hA^{\gamma,\delta'}, &
      (\hA^{\gamma,\gamma'})^{*} &= \hA^{\gamma',\gamma}, &
      (B\otimes B)\hA^{\gamma,\gamma'} &\subseteq
      \hA^{\gamma,\gamma'}, & \langle \hA^{e,e}\hA\rangle &=
      \hA,
    \end{aligned} \\
    \begin{aligned}
      (b\otimes b')\ha &= (\gamma^{-1}(b') \otimes
      \gamma(b))\ha, & \ha (b\otimes b') &= \ha
      (\gamma'{}^{-1}(b') \otimes \gamma'(b))
    \end{aligned}
  \end{gather*}
  for all $\gamma,\gamma',\delta,\delta' \in \Gamma$, $\ha
  \in \hA^{\gamma,\gamma'}$, $b,b'\in B$.  Given such an
  algebra, we write $\hat r$ and $\hat s$ for the
  compositions $B\cong B\otimes 1 \to M(\hA)$ and $B\cong
  1\otimes B \to M(\hA)$, and $\delta_{\ha}:=\gamma$ and
  $\bar\delta_{\ha}:=\gamma'$ whenever $\ha \in
  \hA^{\gamma,\gamma'}$.
\end{definition}
\begin{proposition} \label{proposition:dual-algebra} $\hA$
  has a structure of a
  $(B,\Gamma)^{\ev}$-matrix-algebra, where for all $x,y\in
  A$, $b\in B$,
\begin{align*}
\hat r(b) \hat x &= \widehat{xr(b)},
  & \hat x\hat r(b) &= \widehat{xs(b)},
 & \hat s(b) \hat x &= \widehat{r(b)x},
  & \hat x\hat s(b) &= \widehat{s(b)x}, \\
  \hat y \hat x &= \widehat{x \ast \hat y}, & \hat x^{*} &=
  \widehat{S(x)^{*}}, &   \delta_{\hat x}
  &= \partial_{x}, & \bar\delta_{\hat x} &= \bar\partial_{x}.
\end{align*}
\end{proposition}
\begin{proof}
  The multiplication is associative and turns $\hA$ into an
  algebra by \eqref{eq:convolution-product}.  This algebra
  is non-degenerate because $A \ast \hat A$ spans $A$ by surjectivity of $T_{2}$.

  The $*$-operation is involutive because $\ast \circ S$ is
  involutive, and anti-multiplicative because 
  \begin{align*}
    S(y \ast \hat x)^{*} &=
    \sum    S(y_{(2)}r(\psi(S(x)y_{(1)}))^{*}  \\
    &= \sum S(y_{(2)})^{*}s(\psi(y_{(1)}^{*}S(x)^{*})) \\
    &= \sum S(y)^{*}_{(1)}s(\psi(S(S(y)^{*}_{(2)})S(x)^{*}))
    = S(x)^{*} \ast \widehat{S(y)^{*}}.
  \end{align*}

  For each $b\in B$, the formulas above define multipliers
  $\hat r(b), \hat s(b) \in M(\hA)$ because
  \begin{align*}
    \hat y (\hat r(b)\hat x) = (xr(b) \ast \hat y)^{\widehat{\ }} =
    (x \ast \widehat{ys(b)})^{\widehat{\ }} = (\hat y \hat
    r(b))\hat x
  \end{align*}
  and similarly $\hat y (\hat s(b)\hat x) = (\hat y \hat
  s(b))\hat x$ for all $x,y\in A$ by
  \eqref{eq:convolution-module}.  The maps $\hat r,\hat s
  \colon B\to M(\hA)$ are non-degenerate homo\-mor\-phisms
  because $r,s\colon B\to M(A)$ have the same properties,
  their images evidently commute, and they are involutive
  because
  \begin{align*}
    (\hat x \hat r(b))^{*} = (\widehat{xs(b)})^{*} =
    (S(xs(b))^{*})^{\widehat{\ }} = (S(x)^{*}r(b^{*}))^{\widehat{\ }}
    = \hat r(b^{*})\hat x^{*}
  \end{align*}
  and similarly $(\hat x \hat s(b))^{*} = \hat s(b^{*})\hat
  x^{*} $ for all $x\in A$, $b\in B$.  Furthermore, $\hat r(b)\hat x = \widehat{xr(b)} =
  (r(\gamma(b))x)^{\widehat{\ }} =\hat s(\gamma(b))
  \widehat{x}$ and likewise $\hat x\hat r(b) = \hat x \hat
  s(\gamma'(b)) $ for all $\gamma,\gamma' \in \Gamma$, $x\in
  A_{\gamma,\gamma'}$, $b\in B$.  

  Finally, \eqref{eq:convolution-grading} implies $
  \hA^{\gamma,\gamma'} \hA^{\delta,\delta'} \subseteq
  \delta_{\gamma',\delta} \hA^{\gamma,\delta'}$ for all
  $\gamma,\gamma',\delta,\delta'\in \Gamma$.
\end{proof}
\begin{remark}
  We expect that $\hA$ carries a natural structure of a
  $(B,\Gamma)^{\ev}$-algebra if $A$ has a suitable structure
  of a $(B,\Gamma)^{\ev}$-matrix algebra, and that $\hA$
  carries a natural structure of a $(B,\Gamma)$-Hopf
  $*$-algebroid if $(A,\Delta)$ additionally is
  \emph{$(B,\Gamma)^{\ev}$-bigraded} in the sense that
  \begin{enumerate}
  \item $A$ is the direct sum of the subspaces
    $A^{\delta,\delta'}_{\gamma,\gamma'}:=A^{\delta,\delta'}
    \cap A_{\gamma,\gamma'}$, where
    $\delta,\delta',\gamma,\gamma' \in \Gamma$, and $A=
    \langle A^{e,e}_{e,e}A\rangle$;
  \item $\Delta(A^{\delta,\delta'}_{\gamma,\gamma''})
    \subseteq \sum A^{\alpha,\alpha'}_{\gamma,\gamma'}
    \todot A^{\beta,\beta'}_{\gamma',\gamma''}$ for all
    $\gamma,\gamma',\delta,\delta' \in \Gamma$, where the
    sum is taken over all $\alpha,\alpha',\beta,\beta'\in
    \Gamma$ satisfying $\alpha\beta=\delta$ and
    $\alpha'\beta'=\delta'$.
  \end{enumerate}
  In that case, $ \epsilon(A^{\delta,\delta'}) = 0$ if $
  (\delta,\delta') \neq (e,e)$ and $S(A^{\delta,\delta'})
  \subseteq A^{\delta'{}^{-1},\delta^{-1}}$ for all
  $\delta,\delta' \in \Gamma$. Indeed, For each
  $\delta,\delta' \in \Gamma$, denote by $p^{\delta,\delta'}
  \colon A = \bigoplus_{\gamma,\gamma'} A^{\gamma,\gamma'}
  \to A^{\delta,\delta'} \subseteq A$ the projection,  let
  $\epsilon':=\epsilon \circ p^{e,e} \colon A \to B$, and
  define $S' \colon A \to A$ by $S'|_{A^{\delta,\delta'}} =
  p^{\delta'{}^{-1},\delta^{-1}} \circ
  S|_{A^{\delta,\delta'}}$ for all $\delta,\delta' \in
  \Gamma$. Then one can check that
  $\epsilon'$ is a counit and $S'$ an antipode for
  $(A,\Delta)$ and therefore coincide with $\epsilon$ and
  $S$, respectively.
\end{remark}

\section{Construction of associated measured quantum groupoids}

\label{section:reduced}

Throughout this section, we assume:
\begin{itemize}
\item[\textbf{(A1)}] $(A,\Delta,\epsilon,S,\mu,\phi,\psi)$
  is a measured multiplier $(B,\Gamma)$-Hopf $*$-algebroid.
\end{itemize}
We shall construct operator-algebraic completions of this
algebraic object in the form of a Hopf $C^{*}$-bimodule,
Hopf-von Neumann bimodule and a measured quantum groupoid.
Along the way, we shall impose further assumptions on
$B,\Gamma,\mu,A$ which were  mentioned already in the
introduction, most notably properness of $A$.

The basic idea is to use the GNS-representations for the
weight $\mu$ on the basis $B$ and the functional $\nu$
on the total algebra $A$, respectively.  Naturally, some
restrictions have to be made on $B,\Gamma,\mu$.  To show
that $\nu$ admits a bounded GNS-representation and to lift
the comultiplication to the level of operator algebras, we
use a fundamental unitary.  To take full advantage of this
unitary, we describe its domain and range as relative tensor
products, and show that it is a pseudo-multiplicative
unitary in the sense of \cite{timmermann:cpmu} and
\cite{vallin:pmu}. The necessary modules are introduced in
\S\ref{subsection:modules}, and the unitary itself is
constructed in \S\ref{subsection:unitary}. This part uses
Connes' spatial theory \cite{takesaki:2}, and the relative
tensor product of Hilbert spaces over $C^{*}$-algebras which
was introduced in \cite{timmermann:fiber}.  The fundamental
unitary then gives rise to completions of $A$ and $\hat A$
in the form of Hopf $C^{*}$-bimodules and two Hopf-von
Neumann bimodules; see \S\ref{subsection:bounded}--\S\ref{subsection:hopf-c-bimodules}.To
obtain the full structure of a measured quantum groupoid, we
finally extend the integrals $\phi,\psi$ to the level of
 von Neumann algebras and show that
these extensions are left or right invariant again in
\S\ref{subsection:reduced-extension}.

Before we turn to details, let us briefly sketch the construction of
the fundamental unitary, which we denote by $W$. Its domain
and range can be described as separated completions of the
relative tensor products $\sA\otimesB \rA$ and $\rA \otimesB
\Ar$ with respect to the  sesquilinear forms given by
 \begin{align} \label{eq:pmu-sesquilinear}
   \begin{aligned}
\langle x
     \otimesB y|x' \otimesB y'\rangle_{(\sA\otimesB \rA)} &=
     \nu(x^{*}s(\partial_{y}(\phi(y^{*}y')) )x'), \\
     \langle x \otimesB y)|x' \otimesB y'\rangle_{(\rA\otimesB
       \Ar)} & = \nu(x^{*}r(\phi(y^{*}y'))x'). \\ 
   \end{aligned}
\end{align}
Note that positivity of these forms is not evident because
$\phi$ is not assumed to be completely positive in any
sense.  Given that positivity, the  map
\begin{align*}
T_{4} \colon \rA\otimesB \Ar \to \sA \otimesB \rA, \quad x
  \otimesB y \mapsto \Delta(y)(x \otimesB 1)=\sum y_{(1)}x
  \otimesB y_{(2)},
\end{align*}
extends to a unitary on the respective completions because
it is surjective by Proposition \ref{proposition:galois} and
isometric as shown by the following calculation:
\begin{align} \label{eq:pmu-isometric}
  \begin{aligned}
    \sum \langle y_{(1)}x \otimesB y_{(2)} |y'_{(1)}x' \otimesB
    y'_{(2)}\rangle_{(\sA \otimesB \rA)} &= \sum
    \nu(x^{*}y_{(1)}^{*}s(\partial_{y_{(2)}}(\phi(y_{(2)}^{*}y'_{(2)})))y'_{(1)}x') \\
    &= \sum \nu(x^{*}s( \phi(y_{(2)}^{*}y'_{(2)})
    )y_{(1)}^{*}y'_{(1)}x') \\ &=
    \nu(x^{*}r(\phi(y^{*}y'))x') = \langle x \otimesB y|x'
    \otimesB y'\rangle_{(\rA\otimesB\Ar)}.
  \end{aligned}
\end{align}
The adjoint of this extension is the fundamental unitary $W$.

Similarly, one can construct and employ another unitary $V$
which is an extension of the map $T_{1}\colon \As\otimesB
\sA \to \sA \otimesB \rA$, $x\otimesB y \mapsto
\Delta(x)(1\otimesB y)$.  We shall focus on $W$ because this
unitary is given preference in the theory of locally compact
quantum groups and measured quantum groupoids.

\subsection{Preparations concerning the base} \label{subsection:reduced-prep} We define an
inner product on $B$ by $\langle b|b'\rangle
:=\mu(b^{*}b')$ for all $b,b' \in B$, and denote by $K$ the
Hilbert space obtained by completion, and by $\Lambda_{\mu}
\colon B\to K$ the canonical inclusion.  To proceed, we have
to impose the following  assumption:
\begin{itemize}
\item[\textbf{(A2)}] For each $b\in B$, the following
  equivalent conditions hold:
  \begin{enumerate}
  \item  there exists a
  $K\geq 0$ such that $\mu(c^{*}b^{*}bc)\leq K\mu(c^{*}c)$
  for all $c\in B$;
\item there exists an operator $\pi_{\mu}(b) \in
  \mathcal{L}(K)$ such that
  $\pi_{\mu}(b)\Lambda_{\mu}(c)=\Lambda_{\mu}(bc)$ for all
  $c\in B$.
\end{enumerate}
\end{itemize}
\begin{remark}
  To apply the constructions below, it may be useful to
  first perform a base change, similarly as described in
  \cite[\S2]{timmermann:free}, to replace $B$ by an algebra
  of the form $C_{c}(\Omega)$, where $\Omega$ is a locally
  compact space with an action of $\Gamma$. 
Then condition
  (A2) is automatically satisfied. For example, one can take
  $\Omega$ to be the set of all $*$-homomorphisms $\chi
  \colon B\to \complex$, equipped with the weakest topology
  that makes the function $\Omega \to \complex$, $\chi
  \mapsto \chi(b)$, continuous for each $b\in B$, and
  perform a base change along the canonical map $B \to
  M(C_{c}(\Omega))$. 
  Note, however, that such a base change can not simply be
  applied to left and right integrals, but only to
  bi-integrals.
\end{remark}

Assumption (A2) immediately implies the existence of a
$*$-homomorphism $\pi_{\mu} \colon B\to \mathcal{L}(\Hmu)$
which can be regarded as a GNS-representation for $\mu$.

Recall that a \emph{Hilbert algebra} is a $*$-algebra with
an inner product such that left multiplication by each
element is bounded, the resulting $*$-representation is
non-degenerate, and the involution is pre-closed with respect
to the norm induced by the inner product.  Since $B$ is
commutative, the map $\Lambda_{\mu}(B) \to \Lambda_{\mu}(B)$
given by $\Lambda_{\mu}(b) \mapsto \Lambda_{\mu}(b^{*})$
extends to an anti-unitary operator $J_{\mu}$ on
$K$, and hence $\Lambda_{\mu}(B) \subseteq K$ together
with the $*$-algebra structure inherited from $B$ is a
Hilbert algebra. We thus obtain
\begin{itemize}
\item a von Neumann algebra $N:=\pi_{\mu}(B)'' \subseteq \mathcal{L}(K)$,
\item a n.s.f.\ weight $\tilde \mu$ on $N$ such that $\tilde
  \mu(\pi_{\mu}(b^{*}b)) = \langle
  \Lambda_{\mu}(b)|\Lambda_{\mu}(b)\rangle = \mu(b^{*}b)$
  for all $b\in B$,
\item a left ideal $\frakN_{\tilde \mu} := \{ x\in N :
  \tilde \mu(x^{*}x) <\infty\} \subseteq N$ of
  square-integrable elements,
\item a closed map $\Lambda_{\tilde \mu} \colon
  \frakN_{\tilde \mu} \to K$ such that $(K,\Lambda_{\tilde
    \mu},\Id_{N})$ is a GNS-representation for $\tilde \mu$;
  this is the closure of the map $\pi_{\mu}(B) \to K$ given by
  $\pi_{\mu}(b) \to \Lambda_{\mu}(b)$.
\end{itemize}

\subsection{Various module structures} \label{subsection:modules}

We define an inner product on $A$ by $\langle
a|a'\rangle := \nu(a^{*}a')$ for all $a,a' \in A$, denote by
$H$ the Hilbert space obtained by completion, and by
$\Lambda_{\nu} \colon A\to H$ the canonical inclusion. 

\begin{lemma} \label{lemma:ksgns-phi-psi} There exist maps
  $\Lambda_{\phi},\Lambda_{\psi},\Lambda_{\phi}^{\dag},\Lambda_{\psi}^{\dag}
  \colon A \to \mathcal{L}(\Hmu,\Hnu)$ such that for all
  $x,y \in A$, $b \in B$,
  \begin{align*}
    \Lambda_{\phi}(x)\Lambda_{\mu}(b) &=
    \Lambda_{\nu}(xr(b)), &
    \Lambda_{\phi}(x)^{*}\Lambda_{\nu}(y) &=
    \Lambda_{\mu}(\phi(x^{*}y)), &
    \Lambda_{\phi}(x)^{*}\Lambda_{\phi}(y) &=
    \pi_{\mu}(\phi(x^{*}y)), \\
    \Lambda_{\psi}(x)\Lambda_{\mu}(b) &=
    \Lambda_{\nu}(xs(b)), &
    \Lambda_{\psi}(x)^{*}\Lambda_{\nu}(y) &=
    \Lambda_{\mu}( \psi(x^{*}y)), &
    \Lambda_{\psi}(x)^{*}\Lambda_{\psi}(y) &= \pi_{\mu}(
    \psi(x^{*}y)), \\
    \Lambda_{\phi}^{\dag}(x) \Lambda_{\mu}(b)
    &=\Lambda_{\nu}(r(b)x), &
    \Lambda^{\dag}_{\phi}(x)^{*}\Lambda_{\nu}(y) &=
\Lambda_{\mu}(\phi(y\theta(x^{*}))), &
\Lambda^{\dag}_{\phi}(x)^{*}\Lambda_{\phi}^{\dag}(y) &=
\pi_{\mu}(\phi(y\theta(x^{*}))), \\
    \Lambda_{\psi}^{\dag}(x) \Lambda_{\mu}(b)
    &=\Lambda_{\nu}(s(b)x), &
    \Lambda^{\dag}_{\psi}(x)^{*}\Lambda_{\nu}(y) &=
\Lambda_{\mu}(\psi(y\theta(x^{*}))), &
\Lambda^{\dag}_{\psi}(x)^{*}\Lambda_{\psi}^{\dag}(y) &=
\pi_{\mu}(\psi(y\theta(x^{*}))).
  \end{align*}
\end{lemma}
\begin{proof}
We only prove the assertions concerning $\Lambda_{\phi}$ and
  $\Lambda_{\phi}^{\dag}$. They follow from the relations
  \begin{align*}
    \| \Lambda_{\nu}(xr(b))\|^{2} = \nu(r(b)^{*}x^{*}xr(b))
    &= \mu(b^{*}\phi(x^{*}x)b) \leq
    \|\pi_{\mu}(\phi(x^{*}x))\| \|\Lambda_{\mu}(b)\|^{2}, \\
    \langle \Lambda_{\nu}(y)|\Lambda_{\nu}(xr(b))\rangle =
    \nu(y^{*}xr(b)) &= \mu(\phi(y^{*}x)b) = \langle
    \Lambda_{\mu}(\phi(x^{*}y))|\Lambda_{\mu}(b)\rangle
  \end{align*}
  and
  \begin{align*}
  \| \Lambda_{\nu}(r(b)x)\|^{2} = \nu(x^{*}r(b^{*}b)x) &=
  \nu(\theta^{-1}(x)x^{*}r(b^{*}b)) \\ &=
  \mu(\phi(\theta^{-1}(x)x^{*})b^{*}b)  
  \leq \|\Lambda_{\mu}(b)\|^{2} \|
  \pi_{\mu}(\theta^{-1}(x)x^{*})\|,\\
  \langle \Lambda_{\nu}(y)|\Lambda_{\nu}(r(b)x)\rangle =
  \nu(y^{*}r(b)x) &=
  \nu(\theta^{-1}(x)y^{*}r(b)) \\
 &=
  \mu(\phi(\theta^{-1}(x)y^{*})b) = \langle \Lambda_{\mu}(\phi(y\theta(x^{*})))|\Lambda_{\mu}(b)\rangle,
  \end{align*}
which hold for all
 $x,y\in A$ and $b\in B$. 
\end{proof}
The maps introduced above yield various module structures on
$H$ as follows. Let
\begin{align} \label{eq:subspaces}
  E_{\phi} &:= [\Lambda_{\phi}(A)], &
  E_{\psi} &:= [\Lambda_{\psi}(A)], &
  E_{\phi}^{\dag} &:= [\Lambda_{\phi}^{\dag}(A)], &
  E_{\psi}^{\dag} &:= [\Lambda_{\psi}^{\dag}(A)].
\end{align}

We shall use the following concepts introduced in
\cite{timmermann:fiber,timmermann:cpmu}.  A {\em
  $C^{*}$-$\frakb$-module}, where
$\frakb=(K,[\pi_{\mu}(B)],[\pi_{\mu}(B)])$, consists of a
Hilbert space $L$ and a closed subset $E\subseteq
\mathcal{L}(K,L)$ such that $[EK]=L$, $[E\pi_{\mu}(B)]=E$,
$[E^{*}E]=[\pi_{\mu}(B)]$.  Each such
$C^{*}$-$\frakb$-module gives rise to a normal, faithful,
non-degenerate representation $\rho_{E} \colon
N=\pi_{\mu}(B)'' \to \mathcal{L}(L)$ such that
$\rho_{E}(x)\xi = \xi x$ for all $x\in N$, $\xi \in E$.  A
\emph{$C^{*}$-$(\frakb,\frakb)$-module} is a triple
$(L,E,F)$ such that $(L,E)$ and $(L,F)$ are
$C^{*}$-$\frakb$-modules and $[\rho_{E}(\pi_{\mu}(B))F]=F$
and $[\rho_{F}(\pi_{\mu}(B))E]=[E]$.
\begin{lemma} \label{lemma:pmu-module} The Hilbert space $H$
  is a $C^{*}$-$(\frakb,\frakb)$-module with respect to
  either two of the spaces
  $E_{\phi},E_{\psi},E_{\phi}^{\dag},E_{\psi}^{\dag}$.  The
  representations $\alpha:=\rho_{E_{\phi}^{\dag}}$,
  $\beta:=\rho_{E_{\psi}^{\dag}}$,
  $\halpha:=\rho_{E_{\psi}}$, $\hbeta :=\rho_{E_{\phi}}$ of
  $N$ on $H$ are given by
  \begin{align*} 
    \alpha(\pi_{\mu}(b)) \Lambda_{\nu}(a) &=
    \Lambda_{\nu}(r(b)a), &
    \beta(\pi_{\mu}(b))\Lambda_{\nu}(a) &=
    \Lambda_{\nu}(s(b)a), \\
    \hbeta(\pi_{\mu}(b)) \Lambda_{\nu}(a) &=
    \Lambda_{\nu}(ar(b)), &
    \halpha(\pi_{\mu}(b))\Lambda_{\nu}(a) &=
    \Lambda_{\nu}(as(b)) &&\text{for all } b\in B,a\in A.
  \end{align*}
\end{lemma}
\begin{proof}
  Let $E,F$ be any two of the spaces listed above.  Then
  $[E\Hnu]=\Hnu$ and $[E \pi_{\mu}(B)]=E$ because $\langle
  r(B)s(B) Ar(B)s(B)\rangle = A$, and $[E^{*}E] =
  [\pi_{\mu}(B)]$ because $\phi(A)=B=\psi(A)$. Thus, $(H,E)$
  is a $C^{*}$-$\frakb$-module. The formulas for the associated
  representations are easily verified. Using these formulas
  and the relation $\langle r(B)s(B)Ar(B)s(B)\rangle =A$,
  one easily checks that $[\rho_{E}(\pi_{\mu}(B))F]=F$ and
  $[\rho_{F}(\pi_{\mu}(B))E]=E$.
\end{proof}

Recall that a vector $\zeta$ in a Hilbert space $L$ is
\emph{bounded} with respect to a normal, non-degenerate
representation $\rho \colon N \to L$ and the weight $\tilde
\mu$ if the following equivalent conditions hold:
\begin{enumerate}
\item there exists a $K\geq 0$ such that  $\|\rho(x)\zeta\| \leq K\tilde \mu(x^{*}x)$ for all
  $x\in \frakN_{\tilde \mu}$;
\item there exists an operator $R_{\zeta}^{\rho,\tilde \mu}
  \in \mathcal{L}(K,L)$ such that $R_{\zeta}^{\rho,\tilde
    \mu}\Lambda_{\mu}(x) = \rho(x)\zeta$ for all $x\in
  \frakN_{\tilde \mu}$.
\end{enumerate}
The set of all such bounded vectors is denoted by
$D(L_{\rho},\tilde \mu)$. This spaces carries an $N$-valued
inner product $\langle -|-\rangle_{\rho,\tilde \mu}$, given
by $\langle \zeta|\zeta'\rangle_{\rho,\tilde
  \mu}=(R_{\zeta}^{\rho,\tilde
  \mu})^{*}R_{\zeta'}^{\rho,\tilde \mu}$ for all
$\zeta,\zeta'\in D(L_{\rho},\tilde\mu)$, and $\rho(N)'D(L_{\rho},\tilde \mu)
= D(L_{\rho},\tilde \mu)$ and
\begin{align} \label{eq:bounded-vectors} \Lambda_{\tilde
    \mu}(\langle \zeta|\zeta'\rangle_{\rho,\tilde \mu})
  &=(R_{\zeta}^{\rho,\tilde\mu})^{*}\zeta', & R^{\rho,\tilde
    \mu}_{T\zeta} &= TR^{\rho,\tilde \mu}_{\zeta} && \text{for
    all } T\in \rho(N)', \zeta,\zeta' \in
  D(L_{\rho},\tilde\mu).
\end{align}
\begin{lemma} \label{lemma:module-bounded}
  $\Lambda_{\nu}(A) \subseteq \Da \cap \Db \cap \Dha \cap
  \Dhb$ and for all $x,y\in A$,
  \begin{align*}
    R_{\Lambda_{\nu}(x)}^{\alpha,\tilde\mu} &=
    \Lambda_{\phi}^{\dag}(x), & 
    R^{\beta,\tilde \mu}_{\Lambda_{\nu}(x)} &=
    \Lambda_{\psi}^{\dag}(x), &
    R_{\Lambda_{\nu}(x)}^{\halpha,\tilde \mu} &=
    \Lambda_{\psi}(x), & 
    R_{\Lambda_{\nu}(x)}^{\hbeta,\tilde\mu} &=
    \Lambda_{\phi}(x).
  \end{align*}
\end{lemma}
\begin{proof}
  We shall only prove the assertion concerning $\alpha$. Let $a\in
  A$. Then $\Lambda_{\phi}^{\dag}(a)\Lambda_{\tilde
    \mu}(\pi_{\mu}(b)) = \Lambda_{\nu}(r(b)a) =
  \alpha(\pi_{\mu}(b))\Lambda_{\nu}(a)$ for all $b\in B$,
  and since $\pi_{\mu}(B)$ is a core for $\Lambda_{\tilde
    \mu}$, we can conclude
  $\Lambda_{\phi}^{\dag}(a)\Lambda_{\tilde \mu}(x) =
  \alpha(x)\Lambda_{\nu}(a)$ for all $x\in \frakN_{\tilde
    \mu}$.
\end{proof}
 The preceding result and Lemma \ref{lemma:ksgns-phi-psi}
 imply that for all $x,y \in A$,
 \begin{align} \label{eq:bounded-inner-product}
   \begin{aligned}
     \langle
     \Lambda_{\nu}(x)|\Lambda_{\nu}(y)\rangle_{\alpha,\tilde\mu}
     &= \pi_{\mu}(\phi(y\theta(x^{*}))), & \langle
     \Lambda_{\nu}(x)|\Lambda_{\nu}(y)\rangle_{\beta,\tilde\mu}
     &= \pi_{\mu}(\psi(y\theta(x^{*}))),  \\
     \langle
     \Lambda_{\nu}(x)|\Lambda_{\nu}(y)\rangle_{\hat\alpha,\tilde\mu}
     &= \pi_{\mu}(\psi(x^{*}y)), & \langle
     \Lambda_{\nu}(x)|\Lambda_{\nu}(y)\rangle_{\hat\beta,\tilde\mu}
     &= \pi_{\mu}(\phi(x^{*}y)).
   \end{aligned}
 \end{align}

\subsection{The fundamental unitary}

\label{subsection:unitary}
To define the domain and the range of the fundamental
unitary, we use Connes' relative tensor product of Hilbert
modules and the module structures introduced above.  Connes'
original manuscript on the construction remained
unpublished; we therefore refer to
\cite{takesaki:2} and \cite{timmermann:buch} for details.

The \emph{relative tensor product} $H \swotimes H$ is the
separated completion of the algebraic tensor product $\Db
\otimes K \otimes \Da$ with respect to the sesquilinear form
given by
 \begin{align} \label{eq:rtp-inner}
   \langle \xi \otimes \zeta \otimes \eta|\xi' \otimes
  \zeta' \otimes \eta'\rangle = \langle \zeta|
\langle \xi|\xi'\rangle_{\beta,\tilde \mu}
\langle \eta|\eta'\rangle_{\alpha,\tilde \mu}\zeta'\rangle.
 \end{align}
This Hilbert space can naturally be identified
 with the separated completions of the algebraic tensor
 products $\Db \otimes H$ and $H\otimes \Da$
with respect to the sesquilinear forms given by
\begin{align} \label{eq:rtp-asymmetric}
  \langle \xi \otimes \eta|\xi'\otimes \eta'\rangle &=
  \langle \eta|\alpha(\langle \xi|\xi'\rangle_{\beta,\tilde
    \mu})\eta'\rangle &&\text{and} &
  \langle \xi \otimes \eta|\xi'\otimes \eta'\rangle &=
  \langle \xi|\beta(\langle \eta|\eta'\rangle_{\alpha,\tilde
    \mu})\xi'\rangle,
\end{align}
respectively, via 
\begin{align} \label{eq:rtp-identify}
  \xi \otimes R_{\xi}^{\alpha,\tilde\mu}\zeta \equiv \xi
  \otimes \zeta \otimes \eta \equiv R_{\xi}^{\beta,\tilde
    \mu} \zeta \otimes \eta,
\end{align}
and we shall use these identifications without further
notice.  Replacing the representations $\beta,\alpha$ by
$\alpha,\hbeta$ or $\halpha,\beta$, respectively, one
obtains the relative tensor products $H
\rwotimes H$ and $H\sotimes H$.

To proceed, we shall impose the following
simplifying assumption which essentially says that the cocycle
$(D_{\gamma})_{\gamma}$ in $M(B)$ has a positive square root
on the algebraic level:
\begin{itemize}
\item[\textbf{(A3)}] There exists a  family
  $(D_{\gamma}^{\frac{1}{2}})_{\gamma\in \Gamma}$ in $M(B)$
  such that for all $\gamma,\gamma'\in \Gamma$, $c\in B$,
  \begin{align*}
    D^{\frac{1}{2}}_{e}&=1, &
    (D^{\frac{1}{2}}_{\gamma})^{*} &=
    D^{\frac{1}{2}}_{\gamma}, &
    (D_{\gamma}^{\frac{1}{2}})^{2}
    &=D_{\gamma}, &
    D^{\frac{1}{2}}_{\gamma\gamma'} &=
    \gamma'{}^{-1}(D_{\gamma}^{\frac{1}{2}})D_{\gamma'}^{\frac{1}{2}},
    &
    \mu(c^{*}D^{\frac{1}{2}}_{\gamma}c) \geq 0.
  \end{align*}
\end{itemize}
Clearly, this condition  implies the existence of a unitary
representation $U\colon \Gamma \to \mathcal{L}(\Hmu)$ such
that
 \begin{align} \label{eq:ugamma}
    U_{\gamma}\Lambda_{\mu}(c) &=
  \Lambda_{\mu}(\gamma(cD^{\frac{1}{2}}_{\gamma})), &
U_{\gamma}
  \pi_{\mu}(b)U_{\gamma}^{*} &= \pi_{\mu}(\gamma(b)) &&
  \text{for all } b,c\in B, \gamma \in \Gamma.
  \end{align}

  Similarly as in \eqref{eq:modular-d}, we define linear
  maps $D^{\frac{1}{2}},\bar D^{\frac{1}{2}} \colon A\to A$
  by \begin{align*}
    D^{\frac{1}{2}}(a) &=
    r(D^{\frac{1}{2}}_{\partial^{-1}_{a}})a =
    ar(D^{-\frac{1}{2}}_{\partial_{a}}), & \bar
    D^{\frac{1}{2}}(a) &=
    s(D^{\frac{1}{2}}_{\bar \partial^{-1}_{a}})a =
    as(D^{-\frac{1}{2}}_{\bar\partial_{a}})
  \end{align*}
  for all $a \in A$.  These maps share all the properties of
  the maps $D,\bar D$ listed in Lemma
  \ref{lemma:modular-d}.  Short calculations show that for
  all $x,y\in A$,
  \begin{align} \label{eq:ksgns-inner-phi}
    \Lambda_{\phi}(x)U_{\partial_{x}^{-1}} &=
    \Lambda^{\dag}_{\phi}(D^{\frac{1}{2}}(x)), & \langle
    \Lambda_{\nu}(D^{\frac{1}{2}}(x))|\Lambda_{\nu}(D^{\frac{1}{2}}(y))\rangle_{\alpha,\tilde
      \mu} &= \pi_{\mu}(\partial_{x}(\phi(x^{*}y))),
    \\ \label{eq:ksgns-inner-psi}
    \Lambda_{\psi}(x)U_{\bar\partial_{x}^{-1}} &=
    \Lambda_{\psi}^{\dag}(\bar D^{\frac{1}{2}}(x)), &
    \langle\Lambda_{\nu}(\bar D^{\frac{1}{2}}(x))
    |\Lambda_{\nu}(\bar
    D^{\frac{1}{2}}(y))\rangle_{\beta,\tilde \mu} &=
    \pi_{\mu}(\bar\partial_{x}(\psi(x^{*}y))).
  \end{align}
  Indeed, for all $x,y\in A$ and $b\in B$,
  \begin{gather*}
    \Lambda_{\phi}(x)U_{\partial_{x}^{-1}}\Lambda_{\mu}(b) =
    \Lambda_{\nu}(xr(\partial_{x}^{-1}(bD^{\frac{1}{2}}_{\partial^{-1}_{x}}))
    =
    \Lambda_{\nu}(r(bD^{\frac{1}{2}}_{\partial_{x}^{-1}}x))
    =
    \Lambda_{\psi}^{\dag}(D^{\frac{1}{2}}(x))\Lambda_{\mu}(b), \\
    \Lambda^{\dag}_{\phi}(D^{\frac{1}{2}}(x))^{*}
    \Lambda^{\dag}_{\phi}(D^{\frac{1}{2}}(y)) =
    U_{\partial_{x}^{-1}}^{*}
    \Lambda_{\phi}(x)^{*}\Lambda_{\phi}(y)
    U_{\partial_{y}^{-1}} = U_{\partial_{x}}
    \pi_{\mu}(\phi(x^{*}y)) U_{\partial_{y}^{-1}} =
    \pi_{\mu}(\partial_{x}(\phi(x^{*}y))).
  \end{gather*}
  \begin{lemma} \label{lemma:pmu-domain-range} The
    sesquilinear forms on $\sA\otimesB \rA$ and $\rA
    \otimesB \Ar$ defined in \eqref{eq:pmu-sesquilinear} are
    positive. Denote by $\overline{\sA \otimesB \rA}$ and
    $\overline{\rA \otimesB \Ar}$ the respective separated
    completions. Then there exist isomorphisms
  \begin{align*}
    \begin{aligned}
      \Lambda &\colon \overline{\rA \otimesB \Ar} \to
      H\rwotimes H, & x\otimesB y &\mapsto \Lambda_{\nu}(x)
      \otimesm
      \Lambda_{\nu}(y), \\
      \Lambda' & \colon \overline{\sA \otimesB \rA} \to
      H\swotimes H, & x\otimesB y &\mapsto
      \Lambda_{\nu}(x) \otimesm
      \Lambda_{\nu}(D^{\frac{1}{2}}(y)).
    \end{aligned}
  \end{align*}
\end{lemma}
\begin{proof}
  The maps $\Lambda,\Lambda'$ are surjective because
  $\Lambda_{\nu}(A) \subseteq H$ is dense, and they are
  well-defined and isometric because
  \eqref{eq:rtp-asymmetric},
  \eqref{eq:bounded-inner-product} and
  \eqref{eq:ksgns-inner-phi} imply for all $x,y\in A$
\begin{align*}
  \langle \Lambda(x \otimes y)|\Lambda(x' \otimes y')\rangle
  & =
  \nu(x^{*}s(\phi(y^{*}y'))x'), \\
  \langle \Lambda'(x \otimes y)|\Lambda'(x' \otimes
  y')\rangle &= \nu(x^{*}r( \partial_{y}(\phi(y^{*}y'))
  )x'). \qedhere
\end{align*}
\end{proof}
\begin{proposition} \label{proposition:pmu} There exists a
  unitary $W \colon H \swotimes H \to H \rwotimes H$ such
  that $W^{*} \circ \Lambda = \Lambda'\circ T_{4}$ as maps
  from $\rA \otimesB \Ar$ to $H \swotimes H$, that is, for
  all $x,y \in A$,
  \begin{align*}
    W^{*}(\Lambda_{\nu}(x) \otimesm \Lambda_{\nu}(y)) &=
    \sum \Lambda_{\nu}(\bar D^{\frac{1}{2}}(y_{(1)})x)
    \otimesm \Lambda_{\nu}(y_{(2)}) = \sum
    \Lambda_{\nu}(y_{(1)}x) \otimesm \Lambda_{\nu}(
    D^{\frac{1}{2}}(y_{(2)})), \\ 
    W(\Lambda_{\nu}(x) \otimesm \Lambda_{\nu}(y)) &= \sum
    \Lambda_{\nu}(S^{-1}(D^{-\frac{1}{2}}(y_{(1)}))x)
    \otimesm \Lambda_{\nu}(y_{(2)}) = \sum (\bar
    D^{\frac{1}{2}}(S^{-1}(y_{(1)}))x) \otimesm
    \Lambda_{\nu}(y_{(2)}).
  \end{align*}
 \end{proposition}
 \begin{proof}
   Calculation \eqref{eq:pmu-isometric} and Lemma
   \ref{lemma:pmu-domain-range} imply that the map
   $\Lambda_{\nu}(x) \otimesm \Lambda_{\nu}(y) \mapsto \sum
   \Lambda_{\nu}(y_{(1)}x) \otimesm
   \Lambda_{\nu}(D^{\frac{1}{2}}(y_{(2)}))$
   extends to an isometry $H\rwotimes H \to H\swotimes
   H$. Bijectivity of this isometry and the formula for $W$
   follow from Proposition \ref{proposition:galois}.
\end{proof}
  Similarly, the map $T_{1}$ yields a second fundamental
  unitary:
  \begin{proposition}
  \label{proposition:pmu-v}
  There exists a unitary $V\colon H\sotimes H \to H\rotimes
  H$ such that for all $x,y\in A$,
   \begin{align*} 
 V(\Lambda_{\nu}(x) \otimesm
     \Lambda_{\nu}(y)) &= \sum \Lambda_{\nu}(\bar
     D^{\frac{1}{2}}(x_{(1)})) \otimesm
     \Lambda_{\nu}(x_{(2)}y)  = \sum
     \Lambda_{\nu}(x_{(1)})
     \otimesm \Lambda_{\nu}(D^{\frac{1}{2}}(x_{(2)})y),
 \\
     V^{*}(\Lambda_{\nu}(x) \otimesm \Lambda_{\nu}(y)) &=
     \sum \Lambda_{\nu}(x_{(1)}) \otimesm
     \Lambda_{\nu}(S(\bar D^{-\frac{1}{2}}(x_{(2)}))y)
  \sum
     \Lambda_{\nu}(x_{(1)}) \otimesm
     \Lambda_{\nu}(D^{\frac{1}{2}}(S(x_{(2)}))y).
  \end{align*}
\end{proposition}
\begin{proof}
The formula above defines an isometry $V$ because
  \eqref{eq:rtp-asymmetric},
  \eqref{eq:bounded-inner-product} and
  \eqref{eq:ksgns-inner-psi} imply
  \begin{align*}
    \sum \langle \Lambda_{\nu}(\bar
    D^{\frac{1}{2}}(x_{(1)})) \otimesm
    \Lambda_{\nu}(x_{(2)}y) &|
\Lambda_{\nu}(\bar
    D^{\frac{1}{2}}(x'_{(1)})) \otimesm
    \Lambda_{\nu}(x'_{(2)}y')\rangle_{(H\rotimes H)} \\ &=
    \sum
    \nu(y^{*}x_{(2)}^{*}r(\bar\partial_{x_{(1)}}(\psi(x_{(1)}^{*}x'_{(1)}))x'_{(2)}y')
    \\ &= \sum
    \nu(y^{*}x_{(2)}^{*}x'_{(2)}r(\psi(x_{(1)}^{*}x'_{(1)}))y'),
    \\
 \langle \Lambda_{\nu}(x) \otimesm
    \Lambda_{\nu}(y)|\Lambda_{\nu}(x') \otimesm
    \Lambda_{\nu}(y')\rangle_{(H \sotimes H)} &=
   \nu(y^{*}s(\psi(x^{*}x'))y')
\end{align*}
 for all $x,x',y,y' \in A$, and
by right-invariance of $\psi$ (see Remark
\ref{remarks:algebra-integrals} i)), the expressions above
coincide. Bijectivity of $V$ and the inversion formula
follow from Proposition \ref{proposition:galois}.
\end{proof}

\subsection{Boundedness of the canonical representations}  \label{subsection:bounded}
The first application of the fundamental unitary $W$ is to show that
left multiplication on $A$ and right convolution by $\hat A$
extend to  representations on the Hilbert space $H$.
\begin{theorem} \label{theorem:pmu-bounded}
  There exist $*$-homomorphisms $\pi_{\nu}\colon A \to
  \mathcal{L}(H)$ and $\rho \colon \hat A \to
  \mathcal{L}(H)$ such that
  \begin{align} \label{eq:legs-representations}
    \pi_{\nu}(x)\Lambda_{\nu}(y) &= \Lambda_{\nu}(xy) \text{
      for all } x,y\in A, & \rho(\omega)\Lambda_{\nu}(y)
    &= \Lambda_{\nu}(y\ast \omega) \text{ for all }
    \omega\in \hat A,y\in A.
  \end{align}
\end{theorem}

The proof of Theorem \ref{theorem:pmu-bounded} involves
operators and slice maps of the following form. For each
$\xi \in \Db$ and $\eta \in \Da$, there exist bounded
linear operators
\begin{align} \label{eq:rtp-legs}
  \lambda^{\beta,\alpha}_{\xi} &\colon H \to H\swotimes H, \
  \eta' \mapsto \xi \otimesm \eta', &
  \rho^{\beta,\alpha}_{\eta} &\colon H\to H\swotimes H, \
  \xi' \mapsto \xi' \otimesm \eta,
\end{align}
whose adjoints are given by
\begin{align}\label{eq:rtp-legs-adjoints}
  (\lambda^{\beta,\alpha}_{\xi})^{*}(\xi' \otimes \eta') &=
  \alpha(\langle \xi|\xi'\rangle_{\beta,\tilde \mu})\eta',
  &
  (\rho^{\beta,\alpha}_{\eta})^{*}(\xi' \otimes \eta') &=
  \beta(\langle \eta|\eta'\rangle_{\alpha,\tilde \mu})\xi'.
\end{align}
Likewise, there exist operators
$\lambda^{\alpha,\hbeta}_{\xi},\rho^{\alpha,\hbeta}_{\eta}
\colon H \to H\rwotimes H$ for all $\xi \in \Da$ and $\eta\in
\Dhb$ which are defined similarly.  Using these operators,
one defines  slice maps
\begin{align*}
  \begin{aligned}
    \omega_{\xi,\xi'} \ast \Id &\colon \mathcal{L}(H\rwotimes
    H,H\swotimes H) \to \mathcal{L}(H), & T &\mapsto
    (\lambda_{\xi}^{\beta,\alpha})^{*}T
    \lambda_{\xi'}^{\alpha,\hbeta},  \\
    \Id \ast \omega_{\eta,\eta'} &\colon
    \mathcal{L}(H\rwotimes H,H\swotimes H) \to \mathcal{L}(H),
    & T &\mapsto (\rho_{\eta}^{\beta,\alpha})^{*}T
    \rho_{\eta'}^{\alpha,\hbeta}
  \end{aligned}
\end{align*}
 for all $\xi\in\Db$, $\xi' \in \Da$, $\eta\in
\Da$, $\eta' \in
\Dhb$.
\begin{lemma} \label{lemma:pmu-slice} Let $x,x',y,y' \in
  A$. Then
  \begin{align*}
    (\Id \ast
    \omega_{\Lambda_{\nu}(y),\Lambda_{\nu}(y')})(W^{*})
    \Lambda_{\nu}(x) &= \Lambda_{\nu}(ax), & \text{where }
    a&= \sum \bar D^{-\frac{1}{2}}(y'_{(1)}
    s(\phi(y^{*}y'_{(2)}))), \\
    (\omega_{\Lambda_{\nu}(x),\Lambda_{\nu}(x')} \ast
    \Id)(W^{*}) \Lambda_{\nu}(y) &= \Lambda_{\nu}(y \ast
    \hat c), & \text{where } c &= S^{-1}(\bar
    D^{\frac{1}{2}}(\theta^{-1}(x')x^{*})).
  \end{align*}
\end{lemma}
\begin{proof}
We calculate
    \begin{align*}
  (\rho^{\beta,\alpha}_{\Lambda_{\nu}(y)})^{*}
  W^{*} \rho^{\alpha,\hbeta}_{\Lambda_{\nu}(y')}      \Lambda_{\nu}(x)
&  =   \sum (\rho^{\beta,\alpha}_{\Lambda_{\nu}(y)})^{*} (\Lambda_{\nu}(y'_{(1)}x)
      \otimesm \Lambda_{\nu}
      (D^{\frac{1}{2}}(y_{(2)}')))  && \\
      & = \sum
      \beta(\langle
      \Lambda_{\nu}(y)|\Lambda_{\nu}(D^{\frac{1}{2}}(y'_{(2)})))\rangle_{\alpha,\tilde
        \mu} \Lambda_{\nu}(y'_{(1)}x)  && (\text{Equation }
      \eqref{eq:rtp-legs-adjoints}) \\
      &= \sum
      \Lambda_{\nu}(s(\partial_{y}(\phi(D^{-\frac{1}{2}}(y)^{*}y'_{(2)})))y'_{(1)}x)
      && (\text{Equation } \eqref{eq:ksgns-inner-phi}) \\
      & = \sum
      \Lambda_{\nu}(y'_{(1)}s(\phi(y^{*}D^{-\frac{1}{2}}(y'_{(2)})))x) &&(\text{Lemma \ref{lemma:modular-d}})
       \\
       &=  \sum \Lambda_{\nu}(\bar
       D^{-\frac{1}{2}}(y'_{(1)}s(\phi(y^{*}y'_{(2)})))x), &&
       (\text{Lemma \ref{lemma:modular-d}}), 
\end{align*}
\begin{align*}
  (\lambda^{\beta,\alpha}_{\Lambda_{\nu}(x)})^{*} W^{*}
  \lambda^{\alpha,\hbeta}_{\Lambda_{\nu}(x')}
  \Lambda_{\nu}(y) &= \sum
  (\lambda^{\beta,\alpha}_{\Lambda_{\nu}(x)})^{*}(
  \Lambda_{\nu}(\bar D^{\frac{1}{2}}(y_{(1)})x') \otimesm
  \Lambda_{\nu}( y_{(2)})) &&  \\
  &= \sum \alpha(\langle \Lambda_{\nu}(x) | \Lambda_{\nu}(
  \bar D^{\frac{1}{2}}(y_{(1)})x') \rangle_{\beta,\tilde
    \mu}) \Lambda_{\nu}(y_{(2)}) && (\text{Equation }
  \eqref{eq:rtp-legs-adjoints}) \\
  &= \sum \Lambda_{\nu}(r(\psi(\bar
  D^{\frac{1}{2}}(y_{(1)})x'\theta(x^{*})))y_{(2)})
  &&   (\text{Equation } \eqref{eq:bounded-inner-product}) \\
  & = \sum \Lambda_{\nu}(r(\psi(y_{(1)}\bar
  D^{-\frac{1}{2}}(x'\theta(x^{*}))))y_{(2)}) &&
  (\text{Lemma \ref{lemma:modular-d}}) \\
  &= \sum \Lambda_{\nu}(y_{(2)}r(\psi(\bar
    D^{\frac{1}{2}}(\theta^{-1}(x')x^{*})y_{(1)}))). &&
  (\text{Equation } \eqref{eq:convolution-alt}) \qedhere
    \end{align*}
\end{proof}
\begin{proof}[Proof of Theorem \ref{theorem:pmu-bounded}]
  For the elements $a$ and $c$ of the form in Lemma
  \ref{lemma:pmu-slice}, the maps $\Lambda_{\nu}(y) \mapsto
  \Lambda_{\nu}(ay)$ and $\Lambda_{\nu}(x) \mapsto
  \Lambda_{\nu}(x \ast \hat c)$ coincide with compositions
  of bounded operators and therefore are bounded. Since
  elements of the form like $a,c$ span $A$, we obtain maps
  $\pi_{\nu} \colon A \to \mathcal{L}(H)$ and $\rho
  \colon \hat A \to \mathcal{L}(H)$ satisfying
  \eqref{eq:legs-representations}.  Evidently, $\pi_{\nu}$
  is a $*$-homomorphism.  The map $\rho$ is
  multiplicative by  \eqref{eq:convolution-product} and
  Proposition \ref{proposition:dual-algebra}, and it is
  involutive because by \eqref{eq:convolution-alt-2} and
  Proposition \ref{proposition:dual-algebra},
  \begin{align*}
    \langle \rho(\hat
    x)^{*}\Lambda_{\nu}(z)|\Lambda_{\nu}(y)\rangle &=
    \langle
    \rho(\widehat{S(x)^{*}})\Lambda_{\nu}(z)|\Lambda_{\nu}(y)\rangle \\
    &= \sum \langle
    \Lambda_{\nu}(S(x)^{*}_{(1)}s(\psi(S(S(x)_{(2)}^{*})z)))|\Lambda_{\nu}(y)\rangle
 \\
    &= \nu(s(\psi(z^{*}x_{(1)})S(x_{(2)})y) \\
    &= \nu(z^{*}x_{(1)}s(\psi(S(x_{(2)})y))) = \langle
    \Lambda_{\nu}(z)|\rho(\hat x)\Lambda_{\nu}(y)\rangle. 
    \qedhere
  \end{align*}
\end{proof}
\begin{remarks} \label{remarks:pmu-slice}
  \begin{enumerate}
  \item  $\pi_{\nu}(A)'' \subseteq \hbeta(N)' \cap
    \halpha(N)'$ and $\rho(\hA)'' \subseteq \beta(N)' \cap
    \halpha(N)'$ by \eqref{eq:convolution-module}.
  \item Lemma \ref{lemma:pmu-slice}, Theorem
    \ref{theorem:pmu-bounded} and self-adjointness of
    $\pi_{\nu}(A)$ and $\rho(\hA)$ imply
    \begin{align*}
      \pi_{\nu}(A) &= \mathrm{span} \{ (\Id \ast
      \omega_{\Lambda_{\nu}(y),\Lambda_{\nu}(y')})(W^{*}) |
      y,y' \in A\} =\mathrm{span} \{ (\Id \ast
      \omega_{\Lambda_{\nu}(y),\Lambda_{\nu}(y')})(W) |
      y,y' \in A\}, \\
      \rho(\hA) &= \mathrm{span} \{
      (\omega_{\Lambda_{\nu}(x),\Lambda_{\nu}(x')} \ast
      \Id)(W^{*}) |x,x'\in A\} = \mathrm{span} \{
      (\omega_{\Lambda_{\nu}(x),\Lambda_{\nu}(x')} \ast
      \Id)(W) |x,x'\in A\}.
    \end{align*}
  \end{enumerate}
\end{remarks}
For later use, we calculate the slices of
$V$, which are defined similarly as those of $W^{*}$.
\begin{lemma} \label{lemma:pmu-v-slice}
Let $x,x',y,y' \in A$. Then
  \begin{align*}
(\omega_{\Lambda_{\nu}(x),\Lambda_{\nu}(x')} \ast \Id)(V)
    \Lambda_{\nu}(y) &= \Lambda_{\nu}(ay), & \text{where
    } a &=\sum
    D^{-\frac{1}{2}}(x'_{(2)}r(\psi(x^{*}x'_{(1)}))), \\
(\Id \ast \omega_{\Lambda_{\nu}(y),\Lambda_{\nu}(y')})(V) \Lambda_{\nu}(x)
    &= 
    \Lambda_{\nu}(\check c \ast x), &
    \text{where } c&=S^{-1}(D^{-\frac{1}{2}}(y'\theta(y^{*})).
  \end{align*} 
\end{lemma}
\begin{proof}
  Proceeding as in the proof of that Lemma
  \ref{lemma:pmu-slice}, we find
  \begin{align*}
    (\lambda^{\beta,\alpha}_{\Lambda_{\nu}(x)})^{*} V
    \lambda^{\halpha,\beta}_{\Lambda_{\nu}(x')}
    \Lambda_{\nu}(y) &= \sum
    (\lambda^{\beta,\alpha}_{\Lambda_{\nu}(x)})^{*}
    (\Lambda_{\nu}(\bar D^{\frac{1}{2}}(x'_{(1)}))\otimesm
    \Lambda_{\nu}(x'_{(2)}y)) &&
    (\text{Definition of $V$}) \\
    &=\sum \alpha(\langle
    \Lambda_{\nu}(x)|\Lambda_{\nu}(\bar
    D^{\frac{1}{2}}(x'_{(1)})) \rangle_{\beta,\tilde
      \mu})\Lambda_{\nu}(x'_{(2)}y) \\
    &=\sum \Lambda_{\nu}(r(\bar \partial_{x}(\psi(\bar
    D^{-\frac{1}{2}}(x)^{*}x'_{(1)})))x'_{(2)}y) &&
    \text{(Equation \eqref{eq:ksgns-inner-psi})} \\
    &=\sum
    \Lambda_{\nu}(D^{-\frac{1}{2}}(x'_{(2)}r(\psi(x^{*}x'_{(1)})))y),
  \end{align*}
  \begin{align*}
    (\rho^{\beta,\alpha}_{\Lambda_{\nu}(y)})^{*}V\rho^{\halpha,\beta}_{\Lambda_{\nu}(y')}\Lambda_{\nu}(x)
    &= \sum
    (\rho^{\beta,\alpha}_{\Lambda_{\nu}(y)})^{*}(\Lambda_{\nu}(x_{(1)})\otimesm
    \Lambda_{\nu}(D^{\frac{1}{2}}(x_{(2)})y')) &&
    (\text{Definition of $V$}) \\
    &= \sum \beta(\langle \Lambda_{\nu}(y)|\Lambda_{\nu}
    (D^{\frac{1}{2}}(x_{(2)})
    y')\rangle_{\alpha,\tilde\mu})\Lambda_{\nu}(x_{(1)}) \\
    &= \sum
    \Lambda_{\nu}(s(\phi(D^{\frac{1}{2}}(x_{(2)})y'\theta(y^{*}))x_{(1)})
    && \text{(Equation \eqref{eq:bounded-inner-product})}
    \\
    &= \sum
    \Lambda_{\nu}(s(\phi(x_{(2)}D^{-\frac{1}{2}}(y'\theta(y^{*})))x_{(1)}).
    && \qedhere
  \end{align*}
\end{proof}

\subsection{The Hopf-von Neumann bimodules} \label{subsection:hopf-vn-bimodules}
We next show that the fundamental unitary $W$ is
pseudo-multiplicative in the sense of \cite{vallin:pmu} and
therefore yields two Hopf-von Neumann bimodules, which are
completions of $A$ and $\hat A$, respectively. First, we
need further preliminaries.

The relative tensor product is functorial so that there
exist bounded linear operators $S \otimesm T \in
\mathcal{L}(H\swotimes H)$ for all $S\in \beta(N)', T\in
\alpha(N)'$, as well as $S \otimesm T \in
\mathcal{L}(H\rwotimes H)$ for all $S\in \alpha(N)', T\in
\hbeta(N)'$, both times given by $\xi \otimesm \eta \mapsto
S\xi \otimesm T\eta$. 

In particular, the commuting representations
$\alpha,\beta,\halpha,\hbeta$ yield six representations
$\alpha \otimesm \Id$, $ \halpha\otimesm \Id$, $ \hbeta
\otimesm \Id$, $\Id \otimesm \beta$, $ \Id \otimesm
\halpha$, $\Id\otimesm \hbeta$ of $N$ on $H\swotimes H$, and
further six representations of $N$ on $H\rwotimes H$.
\begin{lemma} \label{lemma:pmu-intertwine}
  The following relations hold  for all $x\in N$:
  \begin{align*} 
W(\Id \otimesm \hbeta(x)) &=     (\beta(x) \otimesm \Id)W, &
W(\halpha(x) \otimesm
    \Id) &=     (\halpha (x) \otimesm \Id)W, &
W(\hbeta(x) \otimesm \Id)     &= (\hbeta(x) \otimesm \Id)W, \\
 W(\alpha(x) \otimesm \Id) &=     (\Id \otimesm \alpha(x))W,
    & W(\Id \otimesm \beta(x)) &= (\Id \otimesm \beta(x))W,
    & W(\Id \otimesm
    \halpha(x)) &= (\Id \otimesm \halpha(x))W.
  \end{align*}
\end{lemma}
\begin{proof}
  This follows immediately from the fact that
  $\pi_{\mu}(B)\subseteq N$ is weakly dense, the definition
  of $W$, and the
  formulas for $\alpha,\beta,\halpha,\hbeta$ given in Lemma
  \ref{lemma:pmu-module}.
\end{proof}

The relative tensor product is associative in a natural
sense.  The intertwining relations for $W$ obtained above
imply that all operators in the diagram below are
well-defined,
\begin{gather} \label{eq:pmu-pentagon} \smalldiagram
  \begin{gathered} \xymatrix@R=20pt@C=20pt{ {\wHone}
      \ar[r]^{W_{12}} \ar[d]^(0.6){W_{23}} &
      {\wHtwo} \ar[r]^{W_{23}} & {\wHthree,} \\
      { H {_{\beta}\otimesm{}_{(\Id \otimesm \alpha)}} (H \rwotimes H)}
      \ar[rr]^{W_{13}} & & {(H \swotimes H) 
        {_{(\alpha \otimesm \Id)}\otimesm{}_{\hbeta}} H}
      \ar[u]^(0.4){W_{12}}}
  \end{gathered}
\end{gather} 
where $W_{12} = W \otimesm \Id$, $W_{23} = \Id \otimesm W$,
and $W_{13}$ acts on the first and third tensor factor; see
\cite{vallin:pmu} for details.
\begin{lemma} \label{lemma:pmu-pentagon}
  Diagram \eqref{eq:pmu-pentagon} commutes, that is,
  $W_{23}W_{12}=W_{12}W_{13}W_{23}$.
\end{lemma}
\begin{proof}
  A short calculation shows that the adjoints of both
  compositions are given by
\begin{align*}
  \Lambda_{\nu}(x) \otimesm \Lambda_{\nu}(y) \otimesm
  \Lambda_{\nu}(z) &\mapsto \sum
  \Lambda_{\nu}(z_{(1)}y_{(1)}x) \otimesm
  \Lambda_{\nu}(D^{\frac{1}{2}}(z_{(2)}y_{(2)})) \otimesm
  \Lambda_{\nu}(D^{\frac{1}{2}}(z_{(3)})). \qedhere
\end{align*}
\end{proof}
\begin{theorem}
  $W$ and $V$ are pseudo-multiplicative unitaries in the sense of
  \cite{vallin:pmu}.
\end{theorem}
\begin{proof}
   The assertion on $W$ is just Lemma
  \ref{lemma:pmu-intertwine} and Lemma
  \ref{lemma:pmu-pentagon}. For $V$, the proof is similar.
\end{proof}

Recall from \cite{vallin:1} that a \emph{Hopf-von Neumann
  bimodule} over $(N,\tilde\mu)$ is a von Neumann algebra
$M$ acting on a Hilbert space $L$ together with faithful,
non-degenerate, commuting normal representations $
\gamma,\delta \colon N\to M$ and a non-degenerate, normal
$*$-homomorphism $\Delta_{M} \colon M \to M
\fibre{\delta}{\tilde \mu}{\gamma} M$ such that $\Delta_{M}
\circ \gamma = \gamma \otimesm \Id$, $\Delta_{M} \circ
\delta = \Id \otimesm \delta$ and $(\Delta_{M}
\fibre{}{\tilde \mu}{} \Id) \circ \Delta_{M} =
(\Id \fibre{}{\tilde \mu}{} \Delta_{M})$, where
$M \fibre{\delta}{\tilde \mu}{\gamma} M = (M' \otimesm M')'
\subseteq \mathcal{L}(L {_{\delta}\otimesm{_{\gamma}}} L)$,
and $\Delta_{M} \fibre{}{\tilde \mu}{} \Id$ and
$\Id \fibre{}{\tilde \mu}{} \Delta_{M}$ are
suitably defined \cite{sauvageot:2}.

Using Remark \ref{remarks:pmu-slice} i) and  slightly abusing notation, we define faithful, normal,
non-degenerate $*$-homomorphisms
\begin{align*}
    \Delta &\colon \pi_{\nu}(A)'' \to
    \mathcal{L}(H\rotimes H), \ x \mapsto W^{*}(\Id
    \otimesm x)W,   \\
    \hDelta &\colon \rho(\hA)'' \to
    \mathcal{L}(H\rwotimes H), \ y \mapsto \Sigma W(y\otimesm
    \Id)W^{*}\Sigma.
\end{align*}
\begin{theorem} \label{theorem:pmu-legs-vn}
  $(\pi_{\nu}(A)'',\alpha,\beta,\Delta)$ and $(\rho(\hat
  A)'',\hbeta,\alpha,\hDelta)$ are Hopf-von Neumann
  bimodules.
\end{theorem}
\begin{proof}
  By Remark \ref{lemma:pmu-slice}, these are the Hopf-von
  Neumann bimodules associated with the
  pseudo-multiplicative unitary $W$; see \cite[\S
  10.3.2]{timmermann:buch}.
\end{proof}
Theorem \ref{theorem:pmu-legs-vn} above can also be
deduced from the following explicit formulas for $\Delta$
and $\hDelta$:
\begin{lemma}
For all $a,c,x,y \in A$,
\begin{align*}
  \Delta(\pi_{\nu}(a))(\Lambda_{\nu}(x) \otimesm
  \Lambda_{\nu}(y)) &= \sum \Lambda_{\nu}(
a_{(1)}x) \otimesm
  \Lambda_{\nu}(  D^{\frac{1}{2}}(a_{(2)})y), \\
      \hDelta(\rho(\hat c))(\Lambda_{\nu}(x) \otimesm
      \Lambda_{\nu}(y)) &= \sum  \Lambda_{\nu}(x_{(2)}r(\psi(S(c)x_{(1)}y_{(1)})))
      \otimesm
      \Lambda_{\nu}(y_{(2)}),
    \end{align*}
\end{lemma}
\begin{proof}
We calculate
  \begin{align*}
    \Delta(\pi_{\nu}(a)) \sum
    \Lambda_{\nu}(y_{(1)}x) \otimesm
    \Lambda_{\nu}(D^{\frac{1}{2}}(y_{(2)}))
&= W^{*}(\Id \otimesm \pi_{\nu}(a))WW^{*}(\Lambda_{\nu}(x)
\otimesm \Lambda_{\nu}(y))
\\ &= W^{*}(\Lambda_{\nu}(x)
\otimesm \Lambda_{\nu}(ay)) \\ &= \sum
\Lambda_{\nu}(a_{(1)}y_{(1)}x) \otimesm \Lambda_{\nu}
(D^{\frac{1}{2}}(a_{(2)}y_{(2)})), 
\end{align*}
\begin{align*}
W^{*}  \hDelta(\rho(\hat c))(\Lambda_{\nu}(x) \otimesm
\Lambda_{\nu}(y)) &=
(\rho(\hat c) \otimesm \Id) W^{*} 
(\Lambda_{\nu}(x) \otimesm
\Lambda_{\nu}(y))
 \\ &= \sum \rho(\hat c)\Lambda_{\nu}(y_{(1)}x) \otimesm \Lambda_{\nu}(D^{\frac{1}{2}}(y_{(2)}))
 \\ &= \sum
 \Lambda_{\nu}(y_{(2)}x_{(2)}r(\psi(S(c)y_{(1)}x_{(1)})))
 \otimesm \Lambda_{\nu}(D^{\frac{1}{2}}(y_{(3)})) \\
 &=W^{*}\sum
 \Lambda_{\nu}(x_{(2)}r(\psi(S(c)y_{(1)}x_{(1)}))) \otimesm
 \Lambda_{\nu}(y_{(2)}). \qedhere
\end{align*}
\end{proof}
\begin{remark}
  Under the identification \eqref{eq:rtp-identify},
   for all $a,x,y\in A$ and $\zeta\in K$
  \begin{align*}
    \Delta(\pi_{\nu}(a))(\Lambda_{\nu}(x) \otimes \zeta
    \otimes \Lambda_{\nu}(y)) &= \sum
    \Lambda_{\nu}(a_{(1)}x) \otimes
    U_{\partial_{a_{(1)}}}\zeta \otimes
    \Lambda_{\nu}(a_{(2)}y).
  \end{align*}
\end{remark}

\subsection{The Hopf $C^{*}$-bimodules}\label{subsection:hopf-c-bimodules}
The fundamental unitary $W$ is regular
$C^{*}$-pseudo-mul\-tiplicative unitaries in the sense of
\cite{timmermann:cpmu}, and therefore yields Hopf
$C^{*}$-bimodules which are completions of $A$ and $\hA$. To
prove this, we again need some preliminaries concerning the
relative tensor product in the setting of $C^{*}$-algebras;
for details, see \cite{timmermann:fiber} and
\cite{timmermann:cpmu}. The construction is parallel to the
von Neumann-algebraic setting and differs mainly in
notation.

As before, let $\frakb=(K,[\pi_{\mu}(B)],[\pi_{\mu}(B)])$.
The relative tensor product $H \fwsource H$ of the
$C^{*}$-$\frakb$-modules $(H,E_{\psi}^{\dag})$ and
$(H,E_{\phi}^{\dag})$ is the separated completion of the
algebraic tensor product $E_{\psi}^{\dag} \otimes K \otimes
E_{\phi}^{\dag}$ with respect to the sesquilinear form given
by
\begin{align} \label{eq:rtp-c-inner}
  \langle \xi \otimes \zeta \otimes \eta | \xi'\otimes
  \zeta' \otimes \eta'\rangle = \langle \zeta|
  (\xi^{*}\xi')(\eta^{*}\eta')\zeta'\rangle.
\end{align}
It can be regarded as a twofold internal tensor product of
Hilbert $C^{*}$-modules and identified with certain
separated completions $E_{\psi}^{\dag} \tr_{\alpha} H$ and
$H_{\beta} \tl E_{\phi}^{\dag}$ of the algebraic tensor
products $E_{\psi}^{\dag} \otimes H$ and $H\otimes
E_{\phi}^{\dag}$, respectively, such that
\begin{align} \label{eq:rtp-c-identify}
  \begin{aligned}
    E_{\psi}^{\dag} \tr {}_{\alpha} H &\cong H\fwsource H\cong
    H_{\beta} \tl E_{\phi}^{\dag}, & \xi \tr \eta\zeta
    &\equiv \xi \otimes \zeta \otimes \eta \equiv \xi \zeta
    \tl \eta.
  \end{aligned}
\end{align}
Comparing the sesquilinear forms \eqref{eq:rtp-inner} with
\eqref{eq:rtp-c-inner} and using \eqref{eq:bounded-inner-product}, one finds that there exists an
isomorphism
\begin{align} \label{eq:rtp-c-vn}
  \begin{aligned}
    H \swotimes H &\cong H \fwsource H, & \Lambda_{\nu}(x)
    \otimes \zeta \otimes \Lambda_{\nu}(y) &\equiv
    \Lambda_{\psi}^{\dag}(x) \otimes \zeta \otimes
    \Lambda^{\dag}_{\phi}(y).
  \end{aligned}
\end{align}
For each $\xi \in E_{\psi}^{\dag}$ and $\eta\in
E^{\dag}_{\phi}$, there exist bounded linear operators
\begin{align*}
  |\xi\rangle_{1} &\colon H\to H \fwsource H, \ \eta' \mapsto
  \xi \tr \eta', & |\eta\rangle_{2} &\colon H\to H\fwsource
  H, \ \xi' \mapsto \xi' \tl \eta.
\end{align*}
We denote their adjoints by $\langle\xi|_{1}$ and
$\langle\eta|_{2}$, respectively, and write
$|E_{\psi}^{\dag}\rangle_{1}=\{ |\xi\rangle_{1} : \xi \in
E_{\psi}^{\dag}\}$, $|E_{\phi}^{\dag}\rangle_{2} =
\{|\eta\rangle_{2} : \eta\in E_{\phi}^{\dag}\}$ et cetera.
Comparing with \eqref{eq:rtp-legs}, we see that under the
identification \eqref{eq:rtp-c-vn},
$\lambda^{\halpha,\beta}_{\Lambda_{\nu}(x)} \equiv
|\Lambda_{\psi}^{\dag}(x)\rangle_{1}$ and
$\rho^{\halpha,\beta}_{\Lambda_{\nu}(y)} \equiv
|\Lambda_{\phi}^{\dag}(y)\rangle_{2}$ for all $x,y\in A$.

Replacing $E^{\dag}_{\psi}$ and $E_{\phi}^{\dag}$ by
$E_{\phi}^{\dag}$ and $E_{\phi}$, respectively, one similarly defines the
relative tensor product $H \fwrange H$ with a canonical
isomorphism $H \fwrange H \cong H\rwotimes H$, and bounded linear operators
$|\xi\rangle_{1},|\eta\rangle_{2}\colon H\to H\fwrange H$ for
all $\xi \in E_{\phi}^{\dag}$ and $\eta\in E_{\phi}$.

Thus, $W$ can be regarded as a unitary  $H\fwsource H\to
H\fwrange H$. To show that it is a
$C^{*}$-pseudo-multiplicative unitary in the sense of
\cite{timmermann:cpmu}, we  only need to prove: 

\begin{proposition} \label{proposition:pmu-intertwine-c}
The following equations for subspaces of
$\mathcal{L}(H,H\frange H)$ hold:
  \begin{align*}
    W[|E_{\psi}^{\dag}\rangle_{1}E_{\phi}] &=
    [|E_{\phi}\rangle_{2}E_{\psi}^{\dag}], &
    W[|E_{\phi}^{\dag}\rangle_{2}E_{\psi}] &=
    [|E_{\phi}\rangle_{2} E_{\psi}], &
    W[|E_{\phi}^{\dag}\rangle_{2}E_{\phi}] &= [|E_{\phi}\rangle_{2}E_{\phi}], \\
    W[|E_{\phi}^{\dag}\rangle_{2}E_{\phi}^{\dag}] &=
    [|E_{\phi}^{\dag}\rangle_{1}E_{\phi}^{\dag}], &
    W[|E_{\psi}^{\dag}\rangle_{1}E_{\psi}^{\dag}] &=
    [|E_{\phi}^{\dag}\rangle_{1}E_{\psi}^{\dag}], &
    W[|E_{\psi}^{\dag}\rangle_{1}E_{\psi}] &=
    [|E_{\phi}^{\dag}\rangle_{1}E_{\psi}].
  \end{align*}
\end{proposition}
The proof uses the following straightforward result:
\begin{lemma} \label{lemma:pmu-intertwine-c} For all
  $x,x',y,y'\in A$ and $\gamma \in \{\alpha,\beta,\hbeta\}$,
  $\gamma' \in \{\alpha,\halpha,\beta\}$,
    \begin{align*}
      \Lambda_{\nu}(x) \otimesm \Lambda_{\nu}(y) &\in D((H
      \swotimes H)_{\Id \otimesm \gamma},\tilde \mu), &
      R^{\Id\otimesm \gamma,\tilde\mu}_{ \Lambda_{\nu}(x)
        \otimesm \Lambda_{\nu}(y)} &=
      \lambda^{\beta,\alpha}_{\Lambda_{\nu}(x)} R^{\gamma,\tilde
        \mu}_{\Lambda_{\nu}(y)}
      =|\Lambda^{\dag}_{\psi}(x)\rangle_{1}R^{\gamma,\tilde
        \mu}_{\Lambda_{\nu}(y)}, \\
      \Lambda_{\nu}(x') \otimesm \Lambda_{\nu}(y') &\in D((H
      \rwotimes H)_{\gamma' \otimesm \Id},\tilde \mu), &
      R^{\gamma'\otimesm \Id,\tilde\mu}_{ \Lambda_{\nu}(x')
        \otimesm \Lambda_{\nu}(y')} &=
      \rho^{\alpha,\hbeta}_{\Lambda_{\nu}(y')}
      R^{\gamma',\tilde \mu}_{\Lambda_{\nu}(x')} =
      |\Lambda_{\phi}(y')\rangle_{2}R^{\gamma',\tilde
        \mu}_{\Lambda_{\nu}(x')}.
    \end{align*}
\end{lemma}
\begin{proof}[Proof of Proposition
  \ref{proposition:pmu-intertwine-c}]
  We only prove the first equation; the others follow
  similarly:
  \begin{align*}
    W[|E_{\psi}^{\dag}\rangle_{1}E_{\phi}] &= [\{ W
    R^{\Id\otimesm \hbeta,\tilde \mu}_{\omega} : \omega \in
    \Lambda_{\nu}(A) \otimesm \Lambda_{\nu}(A)\}] &&
    \text{(Lemma \ref{lemma:pmu-intertwine-c} and
      \eqref{lemma:module-bounded})}  \\
    &= [\{ R^{\beta \otimesm \Id,\tilde \mu}_{W\omega} :
    \omega \in \Lambda_{\nu}(A) \otimesm \Lambda_{\nu}(A)\}]
    &&\text{(Lemma
      \ref{lemma:pmu-intertwine})} \\
    &= [\{ R^{\beta \otimesm \Id,\tilde \mu}_{\omega'} :
    \omega' \in \Lambda_{\nu}(A) \otimesm
    \Lambda_{\nu}(A)\}] && \text{(Definition of $W$)} \\
    &= [|E_{\phi}\rangle_{2}E_{\psi}^{\dag}]. &&
    \text{(Lemma \ref{lemma:pmu-intertwine-c} and
      \ref{lemma:module-bounded})} \qedhere
  \end{align*}
\end{proof}
\begin{theorem}
$W$ and  $V$ are $C^{*}$-pseudo-multiplicative unitaries in the
  sense of \cite{timmermann:cpmu}.
\end{theorem}
\begin{proof}
The assertion on $W$ is Proposition
  \ref{proposition:pmu-intertwine-c} and Lemma
  \ref{lemma:pmu-pentagon}. For $V$, the proof is similar.
\end{proof}
\begin{proposition} \label{proposition:pmu-regular} $W$ and
  $V$ are regular in the sense that $[\langle
  E_{\phi}^{\dag}|_{1}W|E_{\phi}^{\dag}\rangle_{2}]=[E_{\phi}^{\dag}(E_{\phi}^{\dag})^{*}]
  \subseteq \mathcal{L}(\Hnu)$ and $[\langle
  \E|_{1}V|\E\rangle_{2}]=[\E(\E)^{*}] \subseteq
  \mathcal{L}(\Hnu)$.
\end{proposition}
\begin{proof}
  Let $x,x',y \in A$. Then $
  \Lambda_{\phi}^{\dag}(y)\Lambda_{\phi}^{\dag}(x)^{*}
  \Lambda_{\nu}(y') =
  \Lambda_{\nu}(r(\phi(y'\theta(x^{*})))y)$ by Lemma
  \ref{lemma:ksgns-phi-psi} and
    \begin{align*}
      \langle
      \Lambda_{\phi}^{\dag}(y)|_{2}W^{*}|\Lambda_{\phi}^{\dag}(x)\rangle_{1}\Lambda_{\nu}(y')
      &=
      (\rho^{\beta,\alpha}_{\Lambda_{\nu}(y)})^{*}W^{*}(\Lambda_{\nu}(x)
      \otimesm \Lambda_{\nu}(y'))
      \\
      &= \sum \beta(\langle
      \Lambda_{\nu}(y)|\Lambda_{\nu}(D^{\frac{1}{2}}(y'_{(2)}))\rangle_{\alpha,\tilde
      \mu}) \Lambda_{\nu}(y'_{(1)}x) \\
    &= \sum
    \Lambda_{\nu}(s(\phi(D^{\frac{1}{2}}(y'_{(2)})\theta(y^{*})))y'_{(1)}x)
    && (\text{Equation } \eqref{eq:bounded-inner-product})
    \\
    &= \sum \Lambda_{\nu}(s(\phi(y'_{(2)}z))y'_{(1)}x) &&
    \text{with } z:=D^{-\frac{1}{2}}(\theta(y^{*})) \\
    &= \sum
    \Lambda_{\nu}(r(\phi(y'z_{(2)}))S^{-1}(z_{(1)})x). &&
    (\text{Proposition } \ref{proposition:integral-strong-invariance})
  \end{align*}  
  Since the maps $\theta,D^{-\frac{1}{2}},S$ and $T_{3}$ are
  bijections, we can conclude
  \begin{align*}
    [\{ \Lambda_{\phi}^{\dag}(y)\Lambda_{\phi}^{\dag}(x)^{*}
    : x,y\in A\}] &= [\{ \langle
    \Lambda_{\phi}^{\dag}(x)|_{2}W^{*}|\Lambda_{\phi}^{\dag}(y)\rangle_{1}:
    x,y \in A\}]. 
  \end{align*}
  The assertion on $V$ follows from a similar calculation.
\end{proof}

Recall from \cite{timmermann:cpmu} that a
\emph{Hopf $C^{*}$-bimodule over $\frakb$} consists of a
$C^{*}$-$(\frakb,\frakb)$-module $(L,E,F)$, a non-degenerate
$C^{*}$-algebra $C \subseteq \mathcal{L}(L)$  satisfying
$\rho_{E}(\pi_{\mu}(B)) \subseteq M(C)$ and
$\rho_{F}(\pi_{\mu}(B)) \subseteq M(C)$, and a
non-degenerate  $*$-homomorphism $\Delta_{C} \colon C \to C
\fibre{F}{\frakb}{E} C$ that is co-associative and compatible with
$E$ and $F$ in a suitable sense, where 
\begin{align*}
  C \fibre{F}{\frakb}{E} C = \{ T \in \mathcal{L}(L
  {_{F}\underset{\frakb}{\otimes}} {_{E}} L) :
  T|F\rangle_{1}+ T^{*}|F\rangle_{1} \subseteq
  [|F\rangle_{1}C]\text { and } T|E\rangle_{2}+T^{*}|E\rangle_{2} \subseteq
  [|E\rangle_{2}C] \}
\end{align*}
is the \emph{fiber product} of $C$ with itself relative to
$F$ and $E$.

\begin{theorem}
  $\left((H,E_{\phi}^{\dag},E_{\psi}^{\dag}),[\pi_{\nu}(A)],\Delta|_{[\pi_{\nu}(A)]}\right)$
  and
  $\left((H,E_{\psi},E_{\phi}^{\dag}),[\rho(\hA)],\hDelta|_{[\lambda(\hA)]}\right)$
  are Hopf $C^{*}$-bimodules over $\frakb$.
\end{theorem}
\begin{proof}
  By \cite{timmermann:cpmu}, the regular
  $C^{*}$-pseudo-multiplicative unitary $W$ gives rise to
  two Hopf $C^{*}$-bimod\-ules
  $((H,E_{\phi}^{\dag},E_{\psi}^{\dag}), [\langle
  E_{\phi}|_{2}W\kF{2}],\Delta)$ and
  $((H,E_{\psi},E_{\psi}^{\dag}),[\bF{1} W
  \kE{1}],\hDelta)$, and by Lemma \ref{lemma:pmu-slice},
  $[\bF{1} W \kE{1}] = [\rho(\hA)]$ and $[\langle
  E_{\phi}|_{2}W\kF{2}]=[\lambda(\hA)]$.
\end{proof}

\subsection{The measured quantum groupoid}
\label{subsection:reduced-extension}
To obtain a measured quantum groupoid, we finally extend
$\nu,\phi,\psi$ to normal, semi-finite, faithful weights on
the level of von Neumann algebras.  We impose the
following simplifying assumptions:
\begin{itemize}
\item[\textbf{(A4)}] $(A,\Delta)$ is proper in the sense  that 
  $r(B)s(B) \subseteq A$.
\item[\textbf{(A5)}] There exists a net $(u_{i})_{i}$ in $B$
  such that $(\pi_{\mu}(u_{i}))_{i}$ is a net of positive
  elements in the unit ball of $\pi_{\mu}(B)$ that converges
  in $M([\pi_{\mu}(B)])$ strictly to $1$ and such that
  $(\pi_{\mu}(u_{i}^{2}))_{i}$ is increasing.
\end{itemize}
Note that a net $(u_{i})_{i}$ as in (A5)  exists always if
we drop the condition that   $(\pi_{\mu}(u_{i}^{2}))_{i}$
should be increasing.

Let us also note that in the bi-measured case where $\phi,\psi$
and $\nu$ arise from a bi-integral $h$ on $(A,\Delta)$, the
extensions of $\phi,\psi,\nu$ and the invariance of these
extensions can be proved quite easily, see Remark
\ref{remark:biintegral-extension} and
\ref{remark:biintegral-invariant}.

 For the extension of $\nu$,  we do not need the
assumptions (A4) and (A5), but use  the modular
automorphism $\theta$ for $\nu$ obtained in Theorem
\ref{theorem:modular}, the  theory of Hilbert
algebras \cite{takesaki:2}, and results of Kustermans and van
Daele \cite{kustermans:analytic-1}.  
\begin{lemma} \label{lemma:extension-hilbert-algebra}
  $\Lambda_{\nu}(A) \subseteq H$ is a Hilbert algebra with
  respect to the $*$-algebra structure inherited from $A$.
\end{lemma}
\begin{proof}
  The multiplication $\Lambda_{\nu}(y) \mapsto
  \Lambda_{\nu}(xy)$ is bounded for each $x\in A$ by Theorem
  \ref{theorem:pmu-bounded}, and the involution
  $\Lambda_{\nu}(x) \mapsto \Lambda_{\nu}(x^{*})$ is
  pre-closed because
\begin{align*}
    \langle \Lambda_{\nu}(x)|\Lambda_{\nu}(y^{*})\rangle &=
    \nu(x^{*}y^{*}) = \nu(y^{*}\theta(x^{*})) = \langle
    \Lambda_{\nu}(y)| \Lambda_{\nu}(\theta(x^{*}))\rangle
    \quad \text{for all } x,y \in A. \qedhere
  \end{align*}  
\end{proof}

The general theory of Hilbert algebras now yields 
\begin{itemize}
\item  $M=\pi_{\nu}(A)'' \subseteq
  \mathcal{L}(H)$ as the associated von Neumann algebra,
\item a n.s.f.\ weight $\tilde \nu$ on $M$ such that $\tilde
  \nu(\pi_{\nu}(a^{*}a)) = \langle
  \Lambda_{\nu}(a)|\Lambda_{\nu}(a)\rangle = \nu(a^{*}a)$
  for all $a\in A$,
\item a left ideal $\frakN_{\tilde \nu} := \{ x\in M :
  \tilde \nu(x^{*}x) <\infty\} \subseteq M$ of
  square-integrable elements,
\item a closed map $\Lambda_{\tilde \nu} \colon
  \frakN_{\tilde \nu} \to H$ such that $(H,\Lambda_{\tilde
    \nu},\Id_{M})$ is a GNS-representation for $\tilde \nu$;
  this is the closure of the map $\pi_{\nu}(A) \to H$ given by
  $\pi_{\nu}(a) \to \Lambda_{\nu}(a)$;
\item the usual objects
  $J_{\tilde\nu},\Delta_{\tilde\nu},\sigma^{\tilde
    \nu},\mathcal{T}_{\tilde\nu},\ldots$ of Tomita-Takesaki theory.
\end{itemize}

The modular automorphism $\theta$ is related to the modular
automorphism group $\sigma^{\tilde \nu}$ as follows:
\begin{proposition} \label{proposition:weight-modular}
  $\pi_{\nu}(A) \subseteq \mathcal{T}_{\tilde \nu}$ and
  $\sigma^{\tilde
    \nu}_{ni}(\pi_{\nu}(a))=\pi_{\nu}(\theta^{-n}(a))$
  for all $a\in A$, $n \in \integers$.
\end{proposition}
\begin{proof}
  Use the arguments in \cite[\S 3]{kustermans:algebraic}, in
  particular from Lemma 3.16 till Proposition 3.22.
\end{proof}
Let $A^{\theta} :=\{ a\in A: \theta(a)=a\} \subseteq A$.
Note that  this space is a $*$-subalgebra
and, by (A4), contains $r(B)s(B)$.
\begin{lemma} \label{lemma:weight-frakrfraks}
  \begin{enumerate}
  \item $\sigma^{\tilde \nu}$ acts trivially on
    $\pi_{\nu}(A^{\theta})''$, in particular on
    $\alpha(N)$ and $\beta(N)$.
  \item $J_{\tilde \nu}\alpha(x)^{*}J_{\tilde\nu} =
    \hbeta(x)$ and $J_{\tilde\nu}\beta(x)^{*}J_{\tilde\nu} =
    \halpha(x)$ for all $x \in N$.
  \end{enumerate}
\end{lemma}
\begin{proof}
  i) The first assertion follows from the fact that
  $\sigma^{\tilde \nu}_{t}(x) = \Delta_{\tilde \nu}^{it} x
  \Delta_{\tilde \nu}^{-it}$ and $\Delta_{\tilde
    \nu}^{-1}x\Delta_{\tilde \nu} = x$ for each $x\in
  \pi_{\tilde \nu}(A^{\theta})$ by Proposition
  \ref{proposition:weight-modular}, and the second assertion
  follows from the fact that $\sigma^{\tilde \nu}_{t}$ is
  normal for all $t\in \reals$ and acts trivially on
  $\pi_{\nu}(r(B)s(B))$.

  ii) Combine i) and Lemma \ref{lemma:pmu-module}.
\end{proof}
\begin{proposition}
  There exist unique n.s.f.\ weights $T_{L}$ from $M$ to
  $\alpha(N)$ and $T_{R}$ from $M$ to $\beta(N)$ such that $
  \tilde \mu \circ \alpha^{-1} \circ T_{L} = \tilde \nu=
  \tilde \mu \circ \beta^{-1} \circ T_{R}$.
\end{proposition}
\begin{proof}
  This follows from Lemma \ref{lemma:weight-frakrfraks} i) and
  \cite[10.1]{stratila} or \cite[IX
  Theorem 4.18]{takesaki:2}.
\end{proof}

We thus obtain extensions $\tilde \phi := \alpha^{-1} \circ
T_{L}$ and $\tilde \psi:=\beta^{-1} \circ T_{R}$ of $\phi$
and $\psi$.
\begin{remark} \label{remark:biintegral-extension}
  Assume that $\phi=(\Id \otimes \mu) \circ h$ and $\psi =
  (\mu \otimes \Id) \circ h$ for a normalized bi-integral
  $h$ on $(A,\Delta)$. Then the map $\Lambda_{\mu}(B) \otimes
  \Lambda_{\mu}(B) \to \Lambda_{\nu}(A)$ given by
  $\Lambda_{\mu}(b) \otimes \Lambda_{\mu}(b') \mapsto
  \Lambda_{\nu}(r(b)s(b'))$ extends to an isometry $\iota
  \colon K \otimes K \to H$, and a short calculation shows
  that $\iota^{*}\pi_{\nu}(a)\iota = (\pi_{\mu} \otimes
  \pi_{\mu})(h(a))$ for all $a \in A$.  We therefore get a
  positive, normal, linear extension $\tilde h\colon M \to
  N$, $x \mapsto \iota^{*}x\iota$, of $h$, and thereby the
  desired extensions $\tilde \phi = (\Id \bar\otimes \tilde
  \mu) \circ \tilde h$, $\tilde \psi = (\tilde \mu
  \bar\otimes \tilde \mu) \circ \tilde h$ and $\tilde \nu =
  (\tilde \mu \bar\otimes \tilde \mu) \circ \tilde h$.
\end{remark}
 As usual, let $\frakN_{T_{L}}:=\{x\in
M:T_{L}(x^{*}x) \in N\}$ and similarly define $\frakN_{T_{R}}$.
\begin{theorem} \label{theorem:weight-invariance}
  $T_{L}$ and $T_{R}$ are left- and right-invariant with
  respect to $\Delta$ in the sense that
  \begin{align*}
    \tilde
    \phi((\lambda^{\beta,\alpha}_{\xi})^{*}\Delta(x^{*}x)\lambda^{\beta,\alpha}_{\xi})
    &= (R^{\beta,\tilde
      \mu}_{\xi})^{*}T_{L}(x^{*}x)R^{\beta,\tilde \mu}_{\xi}
    \quad \text{for all } x\in \frakN_{T_{L}}, \xi \in \Db,  \\
    \tilde
    \psi((\rho^{\beta,\alpha}_{\eta})^{*}\Delta(x^{*}x)\rho^{\beta,\alpha}_{\eta})
    &= (R^{\alpha,\tilde
      \mu}_{\eta})^{*}T_{R}(x^{*}x)R^{\alpha,\tilde
      \mu}_{\eta} 
    \quad \text{for all } x\in \frakN_{T_{R}}, \eta \in \Da.
  \end{align*}
\end{theorem}
\begin{corollary} \label{corollary:mqg}
  $(N,\tilde\mu,M,\alpha,\beta,\Delta,T_{L},T_{R},\tilde
  \nu)$ is an
  adapted measured quantum groupoid in the sense of
  \cite{lesieur}.
\end{corollary}
To prove Theorem \ref{theorem:weight-invariance}, we 
construct increasing approximations of the weights $\tilde
\mu, \tilde \nu, \tilde \phi,\tilde \psi$ by bounded
positive maps, using an approximate unit $(u_{i})_{i}$ in
$B$ with the properties assumed in (A5).  Let
$u_{i,j}:=r(u_{i})s(u_{j}) \in A$, and define for all $i,j$
bounded, normal, positive, linear maps
\begin{align*}
  \mu_{i} &\colon N \to \complex, \ x\mapsto
  \langle\Lambda_{\mu}(u_{i}) |x\Lambda_{\mu}(u_{i})\rangle, &
  \nu_{i,j} &\colon M \to \complex, \ x \mapsto
  \langle
  \Lambda_{\nu}(u_{i,j})|x\Lambda_{\nu}(u_{i,j})\rangle, \\
  \phi_{i,j} &\colon M \to N, \ x\mapsto
  \Lambda_{\phi}(u_{i,j})^{*}x\Lambda_{\phi}(u_{i,j}), &
  \psi_{i,j} &\colon M \to N, \ x\mapsto
  \Lambda_{\psi}(u_{i,j})^{*} x\Lambda_{\psi}(u_{i,j}).
\end{align*}
Given a net $(\lambda_{\kappa})_{\kappa}$ of real numbers,
we write $(\lambda_{\kappa})_{\kappa} \nearrow \lambda$ if 
it is increasing and converges to $\lambda$. Likewise, given
a von Neumann algebra $C$ with a net
$(\omega_{\kappa})_{\kappa}$ in $C^{+}_{*}$ and a n.s.f.\
weight $\omega$, we write $(\omega_{\kappa})_{\kappa}
\nearrow \omega$ if  $\omega_{\kappa}(x^{*}x) \nearrow
\omega(x^{*}x)$ for all $x\in C$.
\begin{proposition} \label{proposition:weight-truncating}
The following relations hold:
\begin{align*}
  (\mu_{i})_{i} &\nearrow \tilde\mu, & (\nu_{i,j})_{i,j}
  &\nearrow \tilde\nu &&\text{and} & (\upsilon\circ \phi_{i,j})_{i,j}
  &\nearrow \upsilon \circ \tilde \phi, & (\upsilon\circ
  \psi_{i,j})_{i,j} &\nearrow \upsilon \circ \tilde \psi &&\text{for
  all }\upsilon\in N^{+}_{*}.
\end{align*}
\end{proposition}
The proof requires some preparations. We shall focus on the
weights $\tilde \nu$ and $\tilde \phi$; the case $\tilde
\mu$ is quite simple and the case $\tilde \psi$ is similar
to the case $\tilde \phi$.  Recall that an element $\xi \in
H$ is \emph{right-bounded} with respect to the Hilbert
algebra $\Lambda_{\nu}(A)$ if there exists an operator
$R_{\xi} \in \mathcal{L}(H)$ such that
$\pi_{\nu}(a)\xi = R_{\xi} \Lambda_{\nu}(a)$ for all $a \in
A$. Note that then $R_{\xi} \in \frakA'$. Let us call $\xi
\in H$ \emph{right-contractive} if $\xi$ is
right-bounded and $\|R_{\xi}\|\leq 1$. Then $\tilde \nu$ is
given by
\begin{align} \label{eq:weight-extension}
  \tilde \nu(x^{*}x) = \sup \left\{ \|x\xi\|^{2} \,\middle|\, \xi
  \in H \text{ is right-contractive} \right\} \quad \text{for all }
  x\in M.
\end{align}
\begin{lemma} \label{lemma:weight-rightbounded} 
  \begin{enumerate}
  \item If $x \in A^{\theta}$, then the vector $ \Lambda_{\nu}(x) \in
    H$ is right-bounded, $R_{\Lambda_{\nu}(x)} =
    J_{\nu}\pi_{\nu}(x)^{*}J_{\nu}$ and
    $\|R_{\Lambda_{\nu}(x)}\| = \|\pi_{\nu}(x)\|$.
  \item If $x \in A^{\theta} \cap r(B)'$, then
    $\pi_{\nu}(a)\Lambda_{\phi}(x) = R_{\Lambda_{\nu}(x)}
    \Lambda_{\phi}(a)$  for all $a\in A$.
  \item If $a \in A$ and $\xi \in K$ is right-bounded with
    respect to $\Lambda_{\mu}(B)$, then
    $\Lambda_{\phi}(a)\xi =
    \hbeta(R_{\xi})\Lambda_{\nu}(a)$.
  \end{enumerate}
\end{lemma}
\begin{proof}
i)   For all $x\in A^{\theta},a\in A$, we have $\pi_{\nu}(a)
  \Lambda_{\nu}(x) = \Lambda_{\nu}(ax) =
    J_{\nu}\pi_{\nu}(x)^{*}J_{\nu}\Lambda_{\nu}(a)$.

    ii) For all $x \in A^{\theta} \cap r(B)', a\in A,b\in
    B$,
    \begin{align*}
      \pi_{\nu}(a)\Lambda_{\phi}(x)\Lambda_{\mu}(b) =
    \Lambda_{\nu} (axr(b))  &= \Lambda_{\nu}(ar(b)x) \\ &=
    \pi_{\nu}(ar(b))\Lambda_{\nu}(x) = R_{\Lambda_{\nu}(x)}
    \Lambda_{\nu}(ar(b)) = R_{\Lambda_{\nu}(x)}
    \Lambda_{\phi}(a) \Lambda_{\mu}(b).
    \end{align*}

    iii) If $a \in A$ and $\xi=\Lambda_{\mu}(b)$ for some
    $b\in B$, then $R_{\xi}=\pi_{\mu}(b)$ and $\Lambda_{\phi}(a)\xi =
    \Lambda_{\nu}(ar(b)) =
    \hbeta(\pi_{\mu}(b))\Lambda_{\nu}(a)$. Now, the assertion
    follows  for all right-bounded $\xi$ because
    $\Lambda_{\mu}(B)$ is a core for $\Lambda_{\tilde\mu}$
    and the right-bounded elements coincide with $\Lambda_{\tilde\mu}(\frakN_{\tilde\mu})$.
\end{proof}
\begin{proof}[Proof of Proposition \ref{proposition:weight-truncating}]
  We only prove the assertions concerning
  $(\nu_{i,j})_{i,j}$ and $(\phi_{i,j})_{i,j}$; the others
  follow similarly.

  Let $\xi_{i,j}:=\Lambda_{\nu}(u_{i,j})$ and
  $R_{i,j}:=R_{\xi_{i,j}} =
  J_{\nu}\pi_{\nu}(u_{i,j})J_{\nu}$ for all $i,j$.  By Lemma
  \ref{lemma:weight-rightbounded}, each $\xi_{i,j}$ is
  right-contractive and hence $\nu_{i,j}(x^{*}x) =
  \|x\Lambda_{\nu}(u_{i,j})\|^{2} \leq \tilde \nu(x^{*}x)$
  for all $i,j$ and all $x\in M$.  The net
  $(\nu_{i,j})_{i,j}$ in $M_{*}^{+}$ is increasing because
  (i) $(R_{i,j}^{*}R_{i,j})_{i,j}$ is increasing by
  assumption on $(u_{i})_{i}$, (ii)
  $\nu_{i,j}(\pi_{\nu}(a^{*}a)) = \|R_{\xi_{i,j}}
  \Lambda_{\nu}(a)\|^{2}$ for all $a\in A$, and (iii)
  $\pi_{\nu}(A) \subseteq M$ is weakly dense.  For each
  right-contractive $\xi \in H$ and each $x\in M$,
  \begin{align*}
    \|x\xi\|^{2} &= \lim_{i,j} \|x\pi_{\nu}(u_{i,j})\xi\|^{2}
    = \lim_{i,j} \| x R_{\xi} \Lambda_{\nu}(u_{i,j})\|^{2}
    \leq \lim_{i,j} \|x \Lambda_{\nu}(u_{i,j})\|^{2} =
    \lim_{i,j} \nu_{i,j}(x^{*}x)
  \end{align*} 
  because $R_{\xi} \in M'$ and $R_{\xi}^{*}R_{\xi} \leq
  1$. Therefore, $\tilde \nu(x^{*}x) \leq \lim_{i,j}
  \nu_{i,j}(x^{*}x)$.

  A similar argument as above and Lemma
  \ref{lemma:weight-rightbounded} ii) show that for each
  $\upsilon\in N_{*}^{+}$, the net $(\upsilon \circ
  \phi_{i,j})_{i,j}$ is increasing.  Taking  pointwise
  limits, we obtain a normal semi-finite weight $\omega$
  from $M$ to $N$ such that for each
  $y\in M$, the element $\omega(y^{*}y)$ in the
  extended positive part $\hat{N}_{+}$ is
  defined by $\upsilon(\omega(y)) = \sup_{i,j}
  \upsilon(\phi_{i,j}(y^{*}y))$ for all $\upsilon \in
  N_{*}^{+}$. Then for all $y \in M$,
  \begin{align*}
    \tilde \mu(\omega(y^{*}y))\underset{i,j,k}{\nwarrow}
    \| y\Lambda_{\phi}(u_{i,j}) \Lambda_{\mu}(u_{k})\|^{2} =
    \|y \hbeta(\pi_{\mu}(u_{k}))\xi_{i,j}\|^{2}
    \xrightarrow{k\to\infty} \| y \xi_{i,j}\|^{2} = \nu_{i,j}(y^{*}y)
    \underset{i,j}{\nearrow} \tilde \nu(y^{*}y)
  \end{align*}
  and hence $\tilde \mu \circ \omega = \tilde \nu$.  By
  \cite[Theorem 4.18]{takesaki:2}, $\omega=\tilde \phi$.
\end{proof}
The next step towards the proof of Theorem
\ref{theorem:weight-invariance} is the following result:

\begin{lemma} \label{lemma:pmu-flippy}
  $ W^{*}\rho^{\alpha,\hbeta}_{\Lambda_{\nu}(r(b)s(b'))}
  \beta(\pi_{\mu}(b'')) =
  \rho^{\beta,\alpha}_{\Lambda_{\nu}(r(b'')s(b'))}
  \alpha(\pi_{\mu}(b))$ for all $b,b',b'' \in B$.
\end{lemma}
\begin{proof}
  Applying both sides to $\Lambda_{\nu}(a)$, where $a\in A$
  is arbitrary, we obtain $W^{*}(\Lambda_{\nu}(s(b'')a)
  \otimesm \Lambda_{\nu}(r(b)s(b')))$ and $
  \Lambda_{\nu}(r(b) a) \otimesm \Lambda_{\nu}(r(b'')s(b'))
  $, respectively, which coincide.
\end{proof}
\begin{proof}[Proof of Theorem \ref{theorem:weight-invariance}]
To prove the assertion concerning $\tilde \phi$ and
  $T_{L}$, we show that 
  \begin{align} \label{eq:weight-right-invariant}
    \langle \zeta|
    \tilde\phi((\lambda^{\beta,\alpha}_{\xi})^{*}
    \Delta(x^{*}x)\lambda^{\beta,\alpha}_{\xi})\zeta\rangle =
    \|\alpha(\tilde \phi(x^{*}x))^{\frac{1}{2}}
    R^{\beta,\tilde\mu}_{\xi}\zeta\|^{2}
  \end{align}
  for all $x\in \frakN_{T_{L}}$, $\xi \in \Db$ and $\zeta
  \in K$.  Given such $x,\xi,\zeta$, let
  $\xi_{k}:=\alpha(\pi_{\mu}(u_{k}))\xi$ and
  \begin{align*}
    c_{i,j,k} &:= \langle \zeta|
    \phi_{i,j}((\lambda^{\beta,\alpha}_{\xi_{k}})^{*}
    \Delta(x^{*}x)
    \lambda^{\beta,\alpha}_{\xi_{k}}) \zeta\rangle \quad
    \text{for all } i,j,k.
  \end{align*}
  Then $R^{\beta,\tilde\mu}_{\xi_{k}} =
  \alpha(\pi_{\mu}(u_{k}))R^{\beta,\tilde\mu}_{\xi}$,
  $\lambda^{\beta,\alpha}_{\xi_{k}} =
  (\alpha(\pi_{\mu}(u_{k}))\otimesm\Id)\lambda^{\beta,\alpha}_{\xi}$,
  and by Proposition \ref{proposition:weight-truncating},
  \begin{align*}
    c_{i,j,k} \xrightarrow{k\to\infty} \langle \zeta|
    \phi_{i,j}((\lambda^{\beta,\alpha}_{\xi})^{*}
    \Delta(x^{*}x)\lambda^{\beta,\alpha}_{\xi}) \zeta\rangle
    \underset{i,j} \nearrow \langle \zeta|
    \tilde\phi((\lambda^{\beta,\alpha}_{\xi})^{*}
    \Delta(x^{*}x)\lambda^{\beta,\alpha}_{\xi})\zeta\rangle.
  \end{align*}
  On the other hand, using the relation
  $\Lambda_{\phi}(u_{i,j}) =
  \Lambda_{\phi}^{\dag}(u_{i,j})$, we find
  \begin{align*}
    c_{i,j,k} &= \| (1 \btensor x)W\lambda^{\beta,\alpha}_{\xi_{k}}\Lambda_{\phi}(u_{i,j})\zeta\|^{2}
    && \text{(Definition of $\Delta_{W}$ and $\phi_{i,j}$)}
    \\
    &= \| (1 \btensor
    x)W\rho^{\beta,\alpha}_{\Lambda_{\nu}(u_{i,j})}\alpha(\pi_{\mu}(u_k)) 
    R^{\beta,\tilde\mu}_{\xi}\zeta\|^{2} &&
    \text{(Definition of $H\rotimes H$)}\\
    &= \| (1 \btensor
    x)\rho^{\halpha,\beta}_{\Lambda_{\nu}(u_{k,j})}\beta(\pi_{\mu}(u_{i}))
    R^{\beta,\tilde \mu}_{\eta}\zeta\|^{2} && \text{(Lemma
      \ref{lemma:pmu-flippy})} \\ &= \|
    \halpha(\phi_{k,j}(x^{*}x))^{\frac{1}{2}}\beta(\pi_{\mu}(u_{i}))
    R^{\beta,\tilde \mu}_{\xi}\zeta\|^{2}
    \underset{i,j,k}{\nearrow} \|\halpha(\tilde
    \phi(x^{*}x))^{\frac{1}{2}}
    R^{\beta,\tilde\mu}_{\xi}\zeta\|^{2}.  &&
    \text{(Proposition \ref{proposition:weight-truncating})}
  \end{align*}
  Thus, \eqref{eq:weight-right-invariant} follows.  The
  assertion concerning $\tilde \psi$ and $T_{R}$ can be
  proven similarly, where $W$ has to be replaced by the
  unitary $V$.
\end{proof}
\begin{remark} \label{remark:biintegral-invariant} Assume
  that $\phi=(\Id \otimes \mu) \circ h$ for a normalized
  bi-integral $h$ on $(A,\Delta)$. Then for each $b\in B$,
  the map $\Lambda_{\mu}(B)\to \Lambda_{\nu}(A)$ given by
  $\Lambda_{\mu}(c) \mapsto \Lambda_{\nu}(s(b)r(c))$ is
  bounded with norm less than or equal to
  $\mu(b^{*}b)^{\frac{1}{2}}$, and therefore extends to an
  operator $\Lambda_{\phi}(s(b)) \in \mathcal{L}( K,
  H)$. One can then approximate $\tilde \phi$ monotonously
  by the maps $\phi_{i} \colon M\to N$, $x \mapsto
  \Lambda_{\phi}(s(u_{i}))^{*}x\Lambda_{\phi}(s(u_{i}))$,
  and a similar calculation as in Lemma
  \ref{lemma:pmu-flippy} shows that each $\phi_{i}$ is
  right-invariant.
\end{remark}

Associated to the measured quantum groupoid
$(N,\tilde\mu,M,\alpha,\beta,\Delta,T_{L},T_{R},\tilde
\nu)$ are two fundamental unitaries $U'_{H} \colon H
\sotimes H \to H\rotimes H$ and $U_{H} \colon H \rwotimes H
\to H \swotimes H$, characterized by
\begin{align*}
  (\lambda^{\beta,\alpha}_{w})^{*} U_{H} (v \otimesm
  \Lambda_{\tilde\nu}(a)) &=
  \Lambda_{\tilde\nu}((\omega_{w,v} \ast
  \Id)(\Delta(a))) && \text{for all } v,w\in
  D(H_{\beta},\tilde\mu), a \in \frakN_{\tilde \nu}\cap
  \frakN_{T_{L}}, \\
  (\rho^{\beta,\alpha}_{w'})^{*}U'_{H}(\Lambda_{\tilde\nu}(a')
  \otimesm v') &= \Lambda_{\tilde \nu}((\Id \ast
  \omega_{w',v'})(\Delta(a'))), && \text{for all }
  v',w'\in D(H_{\alpha},\tilde\mu), a' \in \frakN_{\tilde
    \nu}\cap \frakN_{T_{R}};
\end{align*}
see \cite[Proposition 3.17]{lesieur}.
\begin{proposition}
$W^{*}=U_{H}$ and   $V=U'_{H}$.
\end{proposition}
\begin{proof}
  Let $x,y,y',z \in A$ and choose $v_{i},w_{i}\in A$ such
  that $\sum \bar D^{\frac{1}{2}}(y_{(1)})x' \otimesB
  y_{(2)} = \sum v_{i} \otimesB w_{i}$ in $\sA\otimesB\rA$. Then
    \begin{align*}
      (\omega_{\Lambda_{\nu}(x),\Lambda_{\nu}(x')} \ast \Id)
      (W^{*}) \Lambda_{\nu}(y) &= \sum_{i}
      (\lambda_{\Lambda_{\nu}(x)}^{\beta,\alpha})^{*}
      (\Lambda_{\nu}(v_{i})
      \otimesm \Lambda_{\nu}(w_{i})) = \sum_{i} \Lambda_{\nu}(r(\psi(v_{i}\theta(x^{*})))w_{i}), \\
      (\omega_{\Lambda_{\nu}(x),\Lambda_{\nu}(x')} \ast
      \Id)(\Delta(y)) \Lambda_{\nu}(z) &=\sum_{i}
      (\lambda_{\Lambda_{\nu}(x)}^{\beta,\alpha})^{*}
      (\Lambda_{\nu}(v_{i}) \otimesm \Lambda_{\nu}(w_{i}z))
      = \sum_{i}
      \pi_{\nu}(r(\psi(v_{i}\theta(x^{*}))))\Lambda_{\nu}(w_{i}z),
    \end{align*}
    and hence $ (\omega_{\Lambda_{\nu}(x),\Lambda_{\nu}(x')}
    \ast \Id) (W^{*})
    \Lambda_{\nu}(y)=\Lambda_{\tilde\nu}((\omega_{\Lambda_{\nu}(x),\Lambda_{\nu}(x')}
    \ast \Id)(\Delta(y)))$. Likewise,  with $v'_{i},w'_{i}
    \in A$ such that $\sum \bar
    D^{\frac{1}{2}}(x_{(1)}) \otimesB x_{(2)}y' = \sum v'_{i}
    \otimesB w'_{i} \in \sA\otimesB \rA$, we find
  \begin{align*}
    (\Id \ast
    \omega_{\Lambda_{\nu}(y),\Lambda_{\nu}(y')})(V)
    \Lambda_{\nu}(x) &= \sum_{i}
    (\rho^{\beta,\alpha}_{\Lambda_{\nu}(y)})^{*}
    (\Lambda_{\nu}(v'_{i}) \otimesm
    \Lambda_{\nu}(w'_{i})) 
    = \sum_{i} \Lambda_{\nu}(s(\phi(w'_{i}\theta(y^{*}))) v'_{i}), \\
    (\Id \ast
    \omega_{\Lambda_{\nu}(y'),\Lambda_{\nu}(y)})(\Delta(\pi_{\nu}(x)))
    \Lambda_{\nu}(z) &= \sum_{i}
    (\rho^{\beta,\alpha}_{\Lambda_{\nu}(y)})^{*}
    (\Lambda_{\nu}(v'_{i}z) \otimesm
    \Lambda_{\nu}(w'_{i}) 
    = \sum_{i} \pi_{\nu}(s(w'_{i}\theta(y^{*})))\Lambda_{\nu}(v'_{i}z)
  \end{align*}
and hence
$(\Id \ast
    \omega_{\Lambda_{\nu}(y),\Lambda_{\nu}(y')})(V)
    \Lambda_{\nu}(x) = \Lambda_{\tilde \nu}\left( (\Id \ast
      \omega_{\Lambda_{\nu}(y'),\Lambda_{\nu}(y)})(\Delta(\pi_{\nu}(x)))\right)$. 
\end{proof}
 The adapted measured quantum groupoid
$(N,\tilde\mu,M,\alpha,\beta,\Delta,T_{L},T_{R},\tilde
\nu)$ has an antipode $\tilde S$ which is characterized by
the following properties:
\begin{enumerate}
\item $\lspan \{(\Id \ast \omega_{v,w} \ast \Id)(V) : w,v \in \mathcal{T}_{\tilde\nu,T_{R}}\}$ is
  a core for $\tilde S$,
\item $ \tilde S((\omega_{w,v} \ast \Id)(V)) =
  (\omega_{w,v} \ast\Id)(V^{*})$ for all $w,v \in
  \mathcal{T}_{\tilde \nu,T_{R}}$,
\end{enumerate}
where $\mathcal{T}_{\tilde\nu,T_{R}}$ is the set of all
$x\in M$ that are analytic with respect to $\sigma^{\tilde \nu}$ and
satisfy $\sigma^{\tilde \nu}_{z} \in \frakN_{\tilde \nu}
\cap \frakN_{\tilde\nu}^{*} \cap \frakN_{T_{R}} \cap
\frakN_{T_{R}}^{*}$ for all $z\in \complex$. Likewise, one
defines $\mathcal{T}_{\tilde \nu,T_{L}}$.
\begin{lemma}
  $\pi_{\nu}(A) \subseteq \mathcal{T}_{\tilde\nu,T_{R}}
  \cap \mathcal{T}_{\tilde\nu,T_{L}}$.
\end{lemma}
\begin{proof}
Recall that   $\pi_{\nu}(A) \subseteq \mathcal{T}_{\tilde\nu}$ by Proposition \ref{proposition:weight-modular}. Using
  Lemma \ref{lemma:weight-frakrfraks} i), we find 
  \begin{align*}
    \sigma^{\tilde\nu}_{z}(\pi_{\nu}(A)) =
    \sigma^{\tilde\nu}_{z}(\pi_{\nu}(As(B))) =
    \sigma^{\tilde \nu}_{z}(\pi_{\nu}(A))
    \beta(\pi_{\mu}(B)) \subseteq \frakN_{\tilde \nu}
    \beta(\frakN_{\tilde \mu}) \subseteq \frakN_{T_{R}}
  \end{align*}
  for all $z\in \complex$. Consequently, $\pi_{\nu}(A)
  \subseteq \mathcal{T}_{\tilde \nu,T_{R}}$. A similar
  argument shows that $\pi_{\nu}(A)
  \subseteq \mathcal{T}_{\tilde \nu,T_{L}}$.
\end{proof}
\begin{proposition}
  $\pi_{\nu}(A) \subseteq \mathrm{Dom}(\tilde S)$ and
  $\tilde S(\pi_{\nu}(a)) = \pi_{\nu}(D^{\frac{1}{2}}SD^{\frac{1}{2}}(a))$ for
  all $a\in A$.
\end{proposition}
\begin{proof}
Let $x,x' \in A$ and $a=\sum
D^{-\frac{1}{2}}(x'_{(2)}r(\psi(x^{*}x'_{(1)})))$. Then
\begin{align*}
  (\omega_{\Lambda_{\nu}(x),\Lambda_{\nu}(x')} \ast \Id)(V)
  &= \pi_{\nu}(a), && \text{(Lemma \ref{lemma:pmu-v-slice})} \\
  (\omega_{\Lambda_{\nu}(x),\Lambda_{\nu}(x')}\ast
  \Id)(V^{*}) &=
  \left((\lambda^{\beta,\alpha}_{\Lambda_{\nu}(x')})^{*}V\lambda^{\halpha,\beta}_{\Lambda_{\nu}(x)}\right)^{*}
  \\
  &= \sum
  \pi_{\nu}(D^{-\frac{1}{2}}(x_{(2)}r(\psi(x'{}^{*}x_{(1)}))))^{*}
  &&
  \text{(Lemma \ref{lemma:pmu-v-slice})} \\
  &= \sum
  \pi_{\nu}(D^{\frac{1}{2}}(r(\psi(x_{(1)}^{*}x'))x^{*}_{(2)}))
  \\
  &= \sum
  \pi_{\nu}(D^{\frac{1}{2}}(S(x'_{(2)}r(\psi(x^{*}x'_{(1)})))))
  && \text{(Proposition
    \ref{proposition:integral-strong-invariance})} \\
  &= \pi_{\nu}(D^{\frac{1}{2}}SD^{\frac{1}{2}}(a)). &&
  \qedhere
\end{align*}
\end{proof}


\bigskip

\emph{Acknowledgments.}
I thank Erik Koelink for introducing
me to dynamical quantum groups and for  stimulating
discussions.

\def\cprime{$'$}

\end{document}